\documentclass[11]{amsart}

\usepackage[margin=1.2in]{geometry}
\usepackage{amsfonts}
\usepackage{amsmath}
\usepackage{amssymb}
\usepackage{color}
\usepackage{soul}

\title[]{An application of  collapsing levels to the representation theory of affine vertex algebras}
\author[Adamovi\'c, Kac, M\"oseneder, Papi, Per\v{s}e]{Dra{\v z}en~Adamovi\'c}
\author[]{Victor~G. Kac}
\author[]{Pierluigi M\"oseneder Frajria}
\author[]{Paolo  Papi}
\author[]{Ozren  Per\v{s}e}

\date{}
\begin{document}
	\def \Z{\mathbb Z}
	\def \C{\mathbb C}
	\def \R{\mathbb R}
	\def \Q{\mathbb Q}
	\def \N{\mathbb N}
	\def \k{\mathfrak k}
	\def \h{\mathfrak h}
	\def \D{\Delta}
	\def \d{\delta}
	\def \half{\tfrac{1}{2}}
	\def \vac{\mathbf 1}
	\def \tr{{\rm tr}}
	\def \span{{\rm span}}
	\def \Res{{\rm Res}}
	\def \End{{\rm End}}
	\def \E{{\rm End}}
	\def \Ind {{\rm Ind}}
	\def \Irr {{\rm Irr}}
	\def \Aut{{\rm Aut}}
	\def \Hom{{\rm Hom}}
	\def \mod{{\rm mod}}
	\def \ann{{\rm Ann}}
	\def \<{\langle}
	\def \>{\rangle}
	\def \t{\tau }
	\def \a{\alpha }
	\def \e{\epsilon }
	\def \l{\lambda }
	\def \L{\Lambda }
	\def \g{\mathfrak g}
	\def \b{\beta }
	\def \om{\omega }
	\def \o{\omega }
	\def \c{\chi}
	\def \ch{\chi}
	\def \cg{\chi_g}
	\def \ag{\alpha_g}
	\def \ah{\alpha_h}
	\def \ph{\psi_h}
	\def \be{\begin{equation}\label}
	\def \ee{\end{equation}}
	\def \bl{\begin{lem}\label}
		\def \el{\end{lem}}
	\def \bt{\begin{thm}\label}
		\def \et{\end{thm}}
	\def \bp{\begin{prop}\label}
		\def \ep{\end{prop}}
	\def \br{\begin{rem}\label}
		\def \er{\end{rem}}
	\def \bc{\begin{coro}\label}
		\def \ec{\end{coro}}
	\def \bd{\begin{de}\label}
		\def \ed{\end{de}}
	\def \pf{{\bf Proof. }}
	\def \voa{{vertex operator algebra}}

	\newcommand{\bea}{\begin{eqnarray}}
	\newcommand{\eea}{\end{eqnarray}}
	
	\newtheorem{thm}{Theorem}[section]
	\newtheorem{prop}[thm]{Proposition}
	\newtheorem{coro}[thm]{Corollary}
	\newtheorem{conj}[thm]{Conjecture}
	\newtheorem{lem}[thm]{Lemma}
	\newtheorem{rem}[thm]{Remark}
	\newtheorem{de}[thm]{Definition}
	\newtheorem{hy}[thm]{Hypothesis}
	\newtheorem{ex}[thm]{Example}
	\makeatletter \@addtoreset{equation}{section}
	\def\theequation{\thesection.\arabic{equation}}
	\makeatother \makeatletter

	\newcommand{\nno}{\nonumber}
	\newcommand{\lbar}{\bigg\vert}
	\newcommand{\p}{\partial}
	\newcommand{\dps}{\displaystyle}
	\newcommand{\bra}{\langle}
	\newcommand{\ket}{\rangle}
	\newcommand{\res}{\mbox{\rm Res}}
	\renewcommand{\hom}{\mbox{\rm Hom}}
	\newcommand{\epf}{\hspace{2em}$\Box$}
	\newcommand{\epfv}{\hspace{1em}$\Box$\vspace{1em}}
	\newcommand{\nord}{\mbox{\scriptsize ${\circ\atop\circ}$}}
	\newcommand{\wt}{\mbox{\rm wt}\ }
	\newcommand \Dim{\rm Dim}

\begin{abstract}
We discover a large class of  simple  affine vertex algebras $V_{k} (\g)$, associated to basic Lie superalgebras $\g$ at  non--admissible collapsing levels $k$,
having exactly   one   irreducible $\g$--locally finite module in the category ${\mathcal O}$.   In the case when $\g$ is a Lie algebra, we prove a complete reducibility result  for  $V_k(\g)$--modules at an arbitrary collapsing level.  We also determine the generators
of the maximal ideal in the universal affine vertex algebra $V^k (\g)$  at certain negative integer levels.  Considering some conformal embeddings  in the simple affine vertex algebras  $V_{-1/2} (C_n)$ and $V_{-4}(E_7)$,  we surprisingly obtain  the realization of non-simple affine vertex algebras  of types $B$ and $D$
  having exactly one   non-trivial ideal.
\end{abstract}

\maketitle

\section{Introduction}
Affine vertex algebras are one of the most interesting and important classes of vertex algebras. Categories
of modules for simple affine vertex algebra $V_{k}(\g)$, associated to a simple Lie algebra $\g$,  have mostly been studied in the case
of positive integer levels $k \in {\mathbb Z}_{\ge 0}$. These categories enjoy many nice properties such as:
finitely many irreducibles, semisimplicity, modular invariance of characters
(cf. \cite{FZ},  \cite{K1}, \cite{KP},   \cite{Z}).

In recent years, affine vertex algebras have attracted a lot of
attention because of their connection with affine ${\mathcal W}$--algebras $W_{k}(\g,f)$, obtained by
quantum Hamiltonian reduction (cf. \cite{FF}, \cite{FKW}, \cite{KRW}, \cite{KW2}). Since the quantum Hamiltonian reduction functor $H_{f} (\ \cdot \ )$
maps any integrable $\widehat{\mathfrak g}$--module to zero (cf. \cite{Araduke}, \cite{KRW}), in order to obtain interesting ${\mathcal W}$--algebras,
one has to consider affine vertex algebras $V_{k}(\g)$, for $k \notin {\mathbb Z}_{\ge 0}$.

It turns out that for certain non-admissible levels $k$ (such as negative integer levels),
the associated vertex algebras $V_{k}(\g)$ have finitely many irreducibles in category
$\mathcal O$  (cf. \cite{AM}, \cite{AM2}, \cite{P}), and their characters satisfy certain modular-like properties (cf. \cite{AK}). These affine
vertex algebras then give $C_2$--cofinite 
${\mathcal W}$--algebras
$W_{k}(\g,f)$, for properly chosen nilpotent element~$f$ (cf. \cite{KW-2008},  \cite{Kaw}).

In this paper, we classify irreducible modules in the category $KL_{k}$ (i.e. the category of   $\g$--locally finite
$V_{k}(\g)$--modules  in ${\mathcal O} ^ k$  (see Subsection \ref{ohr})) for a large family of collapsing levels $k$.
 Recall from \cite{AKMPP-JA} that a level $k$ is called {\it collapsing} if  the simple ${\mathcal W}$--algebra $W_{k}(\g, \theta)$, associated to a minimal nilpotent
element $e_{-\theta}$,  is isomorphic to its affine vertex subalgebra
$\mathcal V_{k} (\g ^{\natural})$ (see Definition \ref{CL} and \eqref{nus}). In the present paper we keep the notation of \cite{AKMPP-JA}. In particular, the highest root is normalized by the condition $(\theta, \theta) = 2$.
We discover a large family of vertex algebras having one irreducible module in the category $KL_{k}$, which
in a way extends the results on Deligne series from \cite{AM}.   Part (1) is proven there in the Lie algebra case.
\begin{thm} \label{thm-classification-unique-introd}
Assume  that  the level $k$ and the basic simple Lie superalgebra $\g$ satisfy one of the following conditions:
\begin{itemize}
\item[(1)]   $k=-\frac{h^\vee}{6}-1$ and $\g$ is one of the Lie algebras of exceptional Deligne's series $A_2$, $G_2$, $D_4$, $F_4$, $E_6$, $E_7$, $E_8$, or $\g=psl(m|m)$ ($m\ge2$), $osp(n+8|n)$ ($n\ge2$), $spo(2|1)$, $F(4)$, $G(3)$ (for both choices of $\theta$);
\item[(2)]  $k =-h^{\vee} / 2 + 1$ and $\g =osp( n + 4 m + 8  | n)$, $n \ge 2, m \ge 0 $.
\item[(3)] $k =-h^{\vee} / 2 + 1$ and $\g = D_{2 m}$, $m \ge 2$.
\item[(4)] $k = -10$ and $\g = E_8$.
\end{itemize}

Then
 $V_k(\g)$ is the unique irreducible $V_k(\g)$--module in the category $KL_k$.
\end{thm}
We also prove  a  complete reducibility result  in $KL_k$ (cf. Theorem \ref{cases-semi-simple}, Theorem \ref{rational}):
\begin{thm}  \label{collapsing-semi-simple}
Assume that $\g$ is a Lie algebra and $k \in {\C} \setminus {\Z}_{\ge 0}$. Then $KL_k$ is a  semi-simple category in the following cases:
\begin{itemize}
\item $k$ is a  collapsing level.
\item $W_k(\g, \theta)$ is a rational vertex operator algebra.
\end{itemize}
\end{thm}
 It is interesting that in some cases we have  that $KL_k$ is a semi-simple category, but there can exist indecomposable  but not irreducible  $V_k(\g)$--modules in the category $\mathcal O$.  In order to prove Theorem \ref{collapsing-semi-simple} we modified methods from \cite{GK} and \cite{DGK}  in a vertex algebraic setting. In particular we prove that the contravariant functor $M \mapsto M^{\sigma}$ from \cite{DGK} acts on  the category $KL_k$ (cf. Lemma \ref{chev}). Then for the proof of  complete reducibility in $KL_k$ it is enough to check that every highest weight $V_k(\g)$--module  in $KL_k$ is irreducible (cf. Theorem  \ref{general-complete-reducibility}).

Representation theory of a  simple affine vertex algebra $V_{k}(\g)$ is naturally connected with
the structure of the  maximal ideal in  the universal affine vertex algebra $V^{k}(\g)$. In the second
part of paper we present explicit formulas for singular vectors which generate the maximal ideal
in $V^{2-2\ell} (D_{2\ell} )$ (which is case (3) of Theorem \ref{thm-classification-unique-introd})
and $V^{-2} (D_{\ell})$. In the second case, we show that the Hamiltonian reduction functor $H_{\theta} (\ \cdot \ )$
gives an equivalence of the category of $\g$--locally finite $V_{-2} (D_{\ell})$--modules $KL_{-2}$ and the category
of modules for a rational vertex algebra $V_{\ell -4} (A_1)$. Singular vectors in $V^{k}(\g)$ for certain
negative integer levels $k$ have also been constructed in~\cite{A}.

We also apply our results to study the structure of conformally embedded subalgebras of some simple affine vertex algebras. \par
  As in \cite{AKMPP-new}, for a  subalgebra $\k$  of a simple Lie algebra $\g$,  we  denote  by   $\widetilde  V (k, \k)$ the vertex subalgebra of $V_k(\g)$ generated by $x(-1) \vac$, $x\in \k$.
If $\k$ is a reductive quadratic  subalgebra of $\g$, then we say
that $\widetilde  V (k, \k)$ is conformally embedded  in $V_k(\g)$  if the Sugawara-Virasoro vectors of both algebras coincide.
We also say that $\k$ is conformally embedded in $\g$ at level $k$ if $\widetilde  V (k, \k)$ is conformally embedded  in $V_k(\g)$.
\par
 We are able to prove that in  the cases listed  in Theorem \ref{thm-confemb-introd} below, $\widetilde  V (k, \k)$ is not simple.
On the other hand, we show that $V_{-1/2}(C_5)$ contains a {\it simple} subalgebra $V_{-2}(B_2) \otimes V_{-5/2}(A_1) $
(see Corollary \ref{coro-simpl}). For the  conformal embedding of  $D_6 \times A_1$ into $E_7$ at level $k = -4$,   we show that   $ \widetilde V (-4 ,D_6 \times A_1 )  =\mathcal{V}_{-4}(D_{6}) \otimes V_{-4}(A_1) $   where  $\mathcal{V}_{-4}(D_{6})$ is a quotient of the universal affine vertex algebra  $V^{-4} (D_6)$ by two singular vectors of conformal weights two and three (cf. \eqref{vmenoquattro}). Moreover,  $\mathcal{V}_{-4}(D_{6})$ has infini\-tely many irreducible modules in the category of $\g$--locally finite modules, which we explicitly describe. All of them appear in $V_{-4}(E_7)$ as submodules or subquotients.
\begin{thm} \label{thm-confemb-introd}
Let $\mathcal V_{k}  (D_{\ell})$, $\mathcal V_k( B_{\ell})$,   be the vertex algebras defined in \eqref{v-2d}, \eqref{v-2b}, \eqref{vmenoquattro}.
 Consider the following conformal embeddings:
\begin{itemize}
\item[(1)] $ D_{\ell} \times A_1 $ into $C_{2l }$  for $\ell \ge 4$ at  level $k =-\frac{1}{2}$.
\item[(2)]  $ B_{\ell} \times A_1 $ into $C_{2l+1  }$ for $\ell \ge 3$  at  level  $k =-\frac{1}{2}$.
\item[(3)]  $D_6 \times A_1$ into $E_7$ at  level $k = -4$.
\end{itemize}
Then,
\begin{itemize}
\item $\widetilde V (-\frac{1}{2}, D_{\ell} \times A_1 ) = \mathcal V_{-2} ( D_{\ell})  \otimes V_{-\ell} (A_1)$ in case (1),
\item  $\widetilde V (-\frac{1}{2}, B_{\ell} \times A_1 )  =  \mathcal V_{-2} ( B_{\ell})  \otimes V_{-\ell-1/2} (A_1)$ in case (2),
\item  $ \widetilde V (-4 ,D_6 \times A_1 ) = \mathcal V_{-4} (D_6) \otimes V_{-4} (A_1)$ in case (3).
\end{itemize}
Moreover, the algebras $\mathcal V_{k}  (D_{\ell})$, $\mathcal V_k( B_{\ell})$,  are non-simple, with a unique non-trivial ideal.
\end{thm}
The decompositions of the embeddings above is still an open problem, and will be a subject of our forthcoming papers.

\vspace{0.5cm}

{\bf Acknowledgement.} We would like to thank Maria Gorelik, Tomoyuki Arakawa and Anne Moreau  for correspondence and discussions.

\section{Preliminaries}

We assume that the reader is familiar with the notion of vertex
(super)algebra (cf. \cite{Bo}, \cite{FLM}, \cite{K2}) and of simple basic   Lie  superalgebras (see \cite{K0}) and their affinizations (see \cite{K1} for the Lie algebra case).

Let $V$ be a conformal vertex algebra. Denote by $A(V)$ the
associative algebra introduced in \cite{Z}, called the Zhu algebra
of $V$.

\subsection{Basic Lie superalgebras and minimal gradings}For the reader's convenience we recall here the setting and notation of \cite{AKMPP-JA} regarding basic Lie superalgebras and their minimal gradings.
Let  $\g=\g_{\bar 0}\oplus \g_{\bar 1}$ be a simple finite dimensional basic   Lie  superalgebra.
We choose a Cartan subalgebra $\h$ for $\g_{\bar 0}$  and let $\D$ be the set of roots.  Assume $\g$ is not $osp(3|n)$.
 A root $-\theta$ is called {\it minimal} if it is even and there exists an additive  function $\varphi:\D\to \R$ such that $\varphi_{|\D}\ne 0$ and $\varphi(\theta)>\varphi(\eta),\,\forall\,\eta\in\D\setminus\{\theta\}$. 
Fix a minimal root $-\theta$ of $\g$. We may choose root vectors $e_\theta$ and $e_{-\theta}$ such that 
$$[e_\theta, e_{-\theta}]=x\in \h,\qquad [x,e_{\pm \theta}]=\pm e_{\pm \theta}.$$
Due to the minimality of $-\theta$, the eigenspace decomposition of $ad\,x$ defines a {\it minimal} $\frac{1}{2}\mathbb Z$-grading  (\cite[(5.1)]{KW2}):
\begin{equation}\label{gradazione}
\g=\g_{-1}\oplus\g_{-1/2}\oplus\g_{0}\oplus\g_{1/2}\oplus\g_{1},
\end{equation}
where $\g_{\pm 1}=\C  e_{\pm \theta}$.   We thus have  a bijective correspondence between minimal gradings (up to an automorphism of $\g$) and minimal roots (up to the action of the Weyl group). Furthermore, one has
\begin{equation}\label{gnatural}
\g_0=\g^\natural\oplus \C x,\quad\g^\natural=\{a\in\g_0\mid (a|x)=0\}.
\end{equation}
Note that  $\g^\natural$ is the centralizer of the triple $\{f_\theta,x,e_\theta\}$.
We can choose $
\h^\natural=\{h\in\h\mid (h|x)=0\},
$ as a  Cartan subalgebra of the Lie superalgebra $\g^\natural$,  so that $\h=\h^\natural\oplus \C x$.\par
For a given choice of a minimal root $-\theta$, we normalize the invariant bilinear form $( \cdot | \cdot)$ on $\g$ by the condition
\begin{equation}\label{normalized}
(\theta | \theta)=2.
\end{equation}
The dual Coxeter number $h^\vee$ of the pair $(\g, \theta)$ (equivalently, of the minimal gradation \eqref{gradazione}) is defined to be   half the eigenvalue of the Casimir operator of $\g$ corresponding to $(\cdot|\cdot)$, normalized  by \eqref{normalized}. Since $\theta$ is the highest root, we have that $2h^\vee=(\theta \vert \theta+2\rho)$ hence
\begin{equation}\label{rhotheta}
(\rho \vert \theta)=h^\vee-1.
\end{equation}

\par
 The complete list of the Lie superalgebras $\g^\natural$, the $\g^\natural$--modules $\g_{\pm 1/2}$ (they are isomorphic and  self-dual),  and $h^\vee$ for all possible choices of $\g$ and of $\theta$ (up to isomorphism)  is given in Tables  1,2,3 of \cite{KW2}. We reproduce them below. Note that in these tables 
 $\g=osp(m|n)$ (resp. $\g=spo(n|m)$) means that $\theta$ is the highest root of the simple component $so(m)$ (resp. $sp(n)$) of $\g_{\bar 0}$. Also, for $\g= sl(m|n)$ or $psl(m|m)$ we always take $\theta$ to be the highest root of the simple component $sl(m)$ of $\g_{\bar 0}$ (for $m=4$ we take one of the simple roots). Note that the exceptional Lie superalgebras $\g=F(4)$ and $\g=G(3)$ appear in both Tables  2 and 3, which corresponds to the two inequivalent choices of $\theta$, the first one being a root of the simple component $sl(2)$ of $\g_{\bar 0}$.
\vskip10pt
{\tiny
\centerline{Table 1}
\vskip 5pt
\noindent {\sl $\g$ is a simple Lie algebra.}
\vskip 5pt
\centerline{\begin{tabular}{c|c|c|c||c|c|c|c}
$\g$&$\g^\natural$&$\g_{1/2}$&$h^\vee$&$\g$&$\g^\natural$&$\g_{1/2}$&$h^\vee$\\
\hline
$sl(n), n\geq 3$&$gl(n-2)$&$\C^{n-2}\oplus (\C^{n-2})^* $&$n$&$F_4$&$sp(6)$&$\bigwedge_0^3\C^6$ & $9$\\\hline 
$so(n), n\geq 5$&$sl(2)\oplus so(n-4)$&$\C^2\otimes\C^{n-4}$&$n-2$&$E_6$&$sl(6)$&$\bigwedge^3\C^6$ & $12$\\\hline
$sp(n), n\geq 2$&$sp(n-2)$&$\C^{n-2} $&$n/2+1$&$E_7$&$so(12)$&$spin_{12}$ & $18$\\\hline
$G_2$&$sl(2)$&$S^3\C^2$&$4$&$E_8$&$E_7$&$\dim=56$ & $30$\\
\end{tabular}}
\vskip 25pt
\centerline{Table 2}
\vskip 5pt
\noindent {\sl $\g$ is not a Lie algebra but $\g^\natural$ is and $\g_{\pm1/2}$ is purely odd ($m\ge1$).}
\vskip 5pt
\centerline{\begin{tabular}{l|c|c|c||c|c|c|c}
$\g$&$\g^\natural$&$\g_{1/2}$&$h^\vee$&$\g$&$\g^\natural$&$\g_{1/2}$&$h^\vee$\\
\hline
$sl(2|m),$&$gl(m)$&$\C^{m}\oplus (\C^{m})^* $&$2-m$&$D(2,1;a)$&{\tiny $sl(2) \oplus sl(2)$}&$\C^2\otimes \C^2$ & $0$\\
$m\ne2$& & & & & & \\\hline
$psl(2|2) $&$sl(2)$&$\C^2\oplus\C^{2}$&$0$&$F(4)$&$so(7)$&$spin_7$ & $-2$\\\hline
$spo(2|m)$&$so(m)$&$\C^{m} $&$2-m/2$&$G(3)$&$G_2$&$ \Dim= 0|7$ & $-3/2$\\\hline
$osp(4|m)$&$sl(2)\oplus sp(m)$&$\C^2\otimes \C^m$&$2-m$\\
\end{tabular}}
\vskip 25pt
\centerline{Table 3}
\vskip 5pt
\noindent {\sl Both $\g$ and $\g^\natural$ are  not  Lie algebras ($m,n\geq 1$).}
\vskip 5pt
\centerline{\begin{tabular}{c|c|c|c}
$\g$&$\g^\natural$&$\g_{1/2}$&$h^\vee$\\
\hline
$sl(m|n)$, $m\neq n, m>2$&$gl(m-2|n)$&$\C^{m-2|n}\oplus(\C^{m-2|n})^*$&$m-n$\\
\hline
$psl(m|m),\,m>2$&$sl(m-2|m)$& $\C^{m-2|m}\oplus(\C^{m-2|m})^*$&$0$\\
\hline
$spo(n|m),\,n\ge 4$& $spo(n-2|m)$ &$\C^{n-2|m}$&$1/2(n-m)+1$\\
\hline
$osp(m|n),\,m\geq 5$&$osp(m-4|n)\oplus sl(2)$ &$\C^{m-4|n}\otimes \C^2$&$m-n-2$\\
\hline
 $F(4)$&$D(2,1;2)$ &$\Dim=6|4$& $3$\\
 \hline
$G(3)$&$osp(3|2)$ &$\Dim=4|4$& $2$\\
\end{tabular}}
}
In this paper we shall exclude the case of $\g=sl(n+2|n)$, $n>0$. In all other cases the Lie superalgebra $\g^\natural$ decomposes in a direct sum of all its minimal ideals, called components of $\g^\natural$:
$$\g^\natural=\bigoplus\limits_{i\in I}\g^\natural_i,
$$
where each  summand is either the (at most 1-dimensional)  center of $\g^\natural$ or is a basic simple Lie superalgebra different from $psl(n|n)$. Let $C_{\g^\natural_i}$  be the Casimir operator of $\g^\natural_i$ corresponding to $(\cdot|\cdot)_{|\g^\natural_i\times \g^\natural_i}$. We define the dual Coxeter number $h^\vee_{0,i}$ of $\g_i^\natural$ as half of the eigenvalue of $C_{\g^\natural_i}$  acting on $\g^\natural_i$ (which is $0$ if $\g_i^\natural$ is abelian).

Denote by $V_{\g} (\mu)$ (or   $V(\mu)$) the
irreducible finite-dimensional highest weight ${\mathfrak g}$--module with highest weight
$\mu$.  Denote by $P_+$ the set of  highest weights of irreducible finite-dimensional representations of  ${\mathfrak g}$.

Since $\h=\h^\natural \oplus \C x$, we have, in particular, that $\mu\in \h^*$ can be uniquely written as
\begin{equation}\label{mumus}
\mu=\mu_{|\h^\natural}+\ell\theta,
\end{equation} with $\ell\in\C$. If $\mu \in P_+$, then, since $\theta(\h^\natural)=0$, $\mu(\theta^\vee)=2\ell\in\Z$, so $\ell\in\half\Z_{\ge0}$.
\par

\subsection{Affine Lie algebras, vertex algebras, $\mathcal W$-algebras}\label{avw}
Let $\widehat{\mathfrak g}$ be the affinization of ${\mathfrak g}$:
$$
\widehat \g=\C[t,t^{-1}]\otimes \g\oplus \C K\oplus \C d
$$
with the usual commutation relations. We let $\d$ be the fundamental imaginary root. Let $\a_0=\d-\theta$ the affine simple root. Since $\theta$ is even, hence non-isotropic, so that $\a_0^\vee=K-\theta^\vee$ makes sense.

Denote by
$L(\lambda)$ (or $L_{{\mathfrak g}}(\lambda)$) the irreducible highest weight $\widehat{\mathfrak
g}$--module with highest weight $\lambda$. 

Denote by $V^{k}(\g)$ the universal affine vertex algebra
associated to $\widehat{\mathfrak g}$ of level $k \in \mathbb{C}$. We shall assume that $k\ne - h^{\vee}$. Then (see e.g. \cite{K2})
$V^{k}(\g)$ is a conformal vertex algebra with Segal-Sugawara conformal vector $\omega_\g$. Let $Y(\omega_\g,z)=\sum L_\g(n)z^{-n-2}$ be the corresponding Virasoro field. Denote by $V_{k}(\g)$ the (unique) simple quotient of $V^{k}(\g)$. Clearly, $V_{k}(\g) \cong L_{{\mathfrak g}}(k \Lambda _0)$
as $\widehat{\mathfrak g}$--modules.

Denote by $W^{k}(\g,\theta)$ the affine ${\mathcal W}$--algebra obtained from  $V^{k}(\g)$ by Hamiltonian reduction
relative to a minimal nilpotent element $e_{-\theta}$. Denote by $W_{k}(\g,\theta)$ the simple quotient of $W^{k}(\g,\theta)$.
 Recall that the vertex algebra ${W}^{k}(\g, \theta)$  is strongly and freely generated by elements  $J^{\{a\}}$, where $a$ runs over a basis of $\g^\natural$, $G^{\{v\}}$, where $v$ runs over a basis of $\g_{-1/2}$, and the
Virasoro vector  $\omega$. The elements $J^{\{a\}},\, G^{\{v\}}$ are primary of conformal weight $1$ and $3/2$,\, respectively, with respect to $\omega$.

Let $\mathcal{V}^k(\g^\natural)$ be the  subalgebra of the vertex algebra ${W}^{k}(\g, \theta)$,  generated by $\{J^{\{a\}}\mid a\in\g^\natural\}$.  The vertex algebra 
$\mathcal{V}^k(\g^\natural)$ is isomorphic to a universal affine vertex algebra. More precisely, letting 
\begin{equation}\label{ki}
k_i=k+\half({h^\vee-h_{0,i}^\vee}),\ i\in I,
\end{equation}
 the map $a\mapsto J^{\{a\}}$ extends   to an isomorphism
$\mathcal{V}^k(\g^\natural)\simeq \bigotimes_{i\in I}V^{k_i}(\g_i^\natural).
$

We also set $\mathcal V_k(\g^\natural)$ to be the image of $\mathcal{V}^k(\g^\natural)$ in ${W}_{k}(\g, \theta)$. Clearly we can write
\begin{equation}\label{nus}
\mathcal V_k(\g^\natural)\simeq \bigotimes_{i\in I} \mathcal V_{k_i}(\g_i^\natural),
\end{equation}
where $\mathcal V_{k_i}(\g_i^\natural)$ is some quotient (not necessarily simple) of $V^{k_i}(\g^\natural_i)$.

\subsection{Category $\mathcal O$ and Hamiltonian reduction functor}\label{ohr}
   Recall that $\widehat  \g$-module $M$ is in  category $\mathcal O ^k $    if it is $\widehat\h$-diagonalizable with finite dimensional weight spaces, $K$ acts as $k Id_M$ and  $M$ has  a finite number of maximal  weights.\par There is a remarkable functor  $H_\theta$ from $\mathcal O ^k $ to the category of  $W^k(\g,\theta)$-modules whose  properties will be very important in the following.
   We recall them in a form suitable for our purposes (see \cite{Araduke} for details; there $H_\theta$ is denoted by $H^0$). 
\begin{thm}\label{A} \ \begin{enumerate}
\item $H_\theta$ is exact.
\item If $L(\l)$ is  a irreducible highest weight  $\widehat\g$-module, then $\l(\a_0^\vee)\in \mathbb Z_{\geq 0}$ implies $H_\theta(L(\l))=\{0\}$. Otherwise $H_\theta(L(\l))$ is isomorphic to the irreducible $W^k(\g,\theta)$-module with highest weight $\phi_{\lambda}$ defined  by formula (67) in \cite{Araduke}.
\end{enumerate}
\end{thm}

\subsection{Collapsing  levels}
\begin{de}\label{CL} Assume $k\ne - h^{\vee}$. If ${W}_{k}(\g, \theta)=\mathcal{V}_{k}(\g^\natural)$, we say that $k$ is  a {\sl collapsing level}.\end{de}

\begin{thm}\label{TT}\cite[Theorem 3.3]{AKMPP-JA} Let $p(k)$ be the polynomial listed in Table 4 below. Then
 $k$  is a collapsing level  if and only if $k\ne-h^\vee$ and
 $p(k)=0$. In such cases, \begin{equation}\label{collapse}
{W}_{k}(\g, \theta)=\bigotimes_{i\in I^*}V_{k_i}(\g^\natural_i),
\end{equation}
where $I^*=\{i\in I\mid k_i\ne 0\}$. If $I^*=\emptyset$, then ${W}_{k}(\g, \theta)=\C$.
\end{thm}

{\small
\centerline{Table 4}}
\vskip 5pt
\noindent {\sl Polynomials $p(k)$.}
\vskip 5pt

{\tiny
\vskip 5pt
\centerline{\begin{tabular}{c|c||c|c}
$\g$&$p(k)$&$\g$&$p(k)$\\
\hline
$sl(m|n)$, $n\ne m$&$(k+1) (k+(m-n)/2)$&$E_6$&$(k+3) (k+4)$\\
\hline
$psl(m|m)$&$ k (k+1)$&$E_7$&$(k+4)(k+6)$\\
\hline
$osp(m|n)$&$(k+2) (k+(m-n-4)/2)$&$E_8$&$(k+6) (k+10)$\\
\hline
$spo(n|m)$&$(k+1/2) (k+(n-m+4)/4)$&$F_4$&$(k+5/2) (k+3)$\\
\hline
$D(2,1;a)$&$(k-a)(k+1+a)$&$G_2$&$ (k+4/3) (k+5/3)$\\
\hline
$F(4)$, $\g^\natural=so(7)$ & $(k+2/3)(k-2/3)$ &$G(3)$, $\g^\natural=G_2$ & $(k-1/2)(k+3/4)$  \\
\hline
$F(4)$, $\g^\natural=D(2,1;2)$ & $(k+3/2)(k+1)$ &$G(3)$, $\g^\natural=osp(3|2)$ & $(k+2/3)(k+4/3)$  \\
\end{tabular}}}
\subsection{Weyl vertex algebra}
Let $M_{\ell}$ denote the Weyl vertex algebra (also called symplectic bosons) generated by even elements $a^{\pm} _i$, $i =1, \dots, \ell$ satisfying the following $\lambda$--brackets
$$[ ( a_i ^{\pm}  )_ {\lambda}  (a_j ^{\pm} ) ] =0, \quad   [(a_i ^{+}  )_{\lambda}  ( a_j ^ {-} ) ]  = \delta_{i,j}. $$
Recall  also that the symplectic affine vertex algebra $V_{-1/2} (C_{\ell})$ is realized as a $\Z_2$--orbifold of $M_{\ell}$ (see \cite{FeF}).

 \section{ The category $KL_k$} \label{Section 2}

 Let $k$ be  a noncritical level. Note that the Casimir element  of $\widehat \g$ can be expressed as $\Omega = d + L_{\g} (0)$;  it commutes with $\widehat{\mathfrak g}$--action.
  
 Consider the category $\mathcal C^k$ of modules for the  universal affine vertex algebra $V^{k}(\g)$, i.e. the category of restricted  $\widehat \g$--modules of level $k$. Regard $M\in \mathcal C^k$ as a
   $\widehat{\mathfrak g}$--module by letting  $d$ act as $-L_{\g} (0)$. 
   Let  $KL^k$ be the category of modules  $M\in \mathcal C^k$ such that, as $\widehat \g$-modules, are in $\mathcal O^k $ and which  admit the following weight  space decomposition with respect to $L_\g(0)$:
$$ M = \bigoplus_{\alpha \in {\C} } M (\alpha), \quad L_{\g} (0) \vert M(\alpha) \equiv \alpha\, \mbox{Id}, \  \dim M(\alpha) < \infty.  $$
 Our definition is related but different from the one introduced in \cite{Ara-2015}. 
Let $KL_k$ be the category of all modules in $KL^k$ which are $V_k(\g)$--modules.

 \begin{rem}
  If  $V_k(\g)$  has finitely many irreducible modules in the category $KL^k$, one can show that every $V_k(\g)$--module $M$ in $KL_k$ is of finite length. This happens when $k$ is admissible (cf.  \cite{Araduke}) and  when $V_k(\g)$ is quasi-lisse (cf. \cite{AK}).
  But when $V_k(\g)$ has infinitely many irreducible modules in $KL^k$ (as in  the cases considered in   \cite{KL}, \cite{AP}),  then  one can have modules in $KL_k$ of infinite length.   \end{rem}
 Recall that there is a one-to-one correspondence between irreducible ${\Z}_{\ge 0}$--graded modules for  a  conformal vertex algebra $V$ (with a conformal vector $\omega$, such that $Y(\omega, z) =\sum_{i\in\mathbb Z} L(i) z^{-i-2}$)    and irreducible modules for the corresponding Zhu algebra $A(V)$ \cite{Z}. This  implies, in particular,  that there is a one-to-one correspondence between irreducible finite-dimensional $A(V)$--modules and irreducible ${\Z}_{\ge 0}$--graded $V$--modules whose graded components,  which are eigenspaces for $L(0)$, are finite-dimensional.
  In the case of affine vertex algebras, we have the following simple interpretation.
 \begin{prop} \label{zhu-int} Let $\widetilde V_k(\g)$ be a quotient of $V^k(\g)$ (not necessary simple).  Consider  $\widetilde V_k(\g)$ as a conformal vertex algebra with conformal vector $\omega_{\g}$. Then
   there is a  one-to-one correspondence between irreducible  $\widetilde V_k(\g)$ in the category $KL^k$  and irreducible finite-dimensional $A(\widetilde V_k(\g))$--modules.
 \end{prop}
 \begin{coro}
 \label{simplicity-1}
 Assume that $\g$ is a simple {\it basic } Lie    superalgebra and   $\widetilde V_k(\g)$ is a quotient of $V^k(\g)$ such that  the  trivial module ${\C}$ is the unique  finite-dimensional irreducible  $A(\widetilde V_k(\g))$--module. Then $\widetilde V_k(\g) = V_k(\g)$.
 \end{coro}
 \begin{proof}
Assume that $\widetilde V_k(\g)$ is not simple. Then it contains a   non-zero graded ideal $I \ne \widetilde V_k(\g)$ with respect to $L_{\g}(0)$:
$$ I = \bigoplus_{n \in  {\Z}_{\ge 0}  }  I(n+n_0), \quad L_{\g} (0) \vert I(r)  = r \mbox{Id}, \quad I(n_0) \ne 0.$$
Since $I \ne \widetilde V_k(\g)$, we have that $n_0 >0$, otherwise ${\bf 1} \in I$.

We can consider $I(n_0)$ as a finite-dimensional module for $\g$ and for the Zhu algebra $A(\widetilde V_k(\g))$.

Since the Casimir element $C_{\g}$  of $\g$ acts  on $I(n_0)$ as the  non-zero constant    $ 2 (k+h ^{\vee}) n_0 $, we conclude that $C_{\g}$ acts  by the same constant on any irreducible $\g$--subquotient of $I(n_0)$. But  any irreducible subquotient of $I(n_0)$ is  an irreducible finite--dimensional $A(\widetilde V_k(\g))$--module, and therefore it is   trivial.
This implies that $C_{\g}$ acts non-trivially on a trivial $\g$--module, a contradiction.
 \end{proof}
 Take the  Chevalley generators  $e_i, f_i, h_i$, $i=0, \dots, \ell$, of the Kac--Moody Lie algebra $\widehat{\g}$  such that $e_i, f_i, h_i$, $i=1, \dots, \ell$,  are the Chevalley generators  of $\g$.
 Let $\sigma$ be the  Chevalley antiautomorphism of $\widehat{\g}$ defined  by
  $$ e_i  \mapsto  f_i  , \ \ f_i   \mapsto    e_i , \ h_i  \mapsto  h_i,\ \  d\mapsto d  \quad (i = 0, \dots, \ell).$$

 Assume that $M$ is  from  the category ${\mathcal O}$ of non-critical  level $k$. Then  $M$  admits the decomposition into  weight spaces
 $ M = \bigoplus_{\mu \in \Omega(M)} M_{\mu}$, where $\Omega(M)$ is the set of weights of $M$ and $\dim M_{\mu} < \infty$ for every $\mu \in \Omega(M)$. For a finite-dimensional vector spaces $U$, let $U^*$ denote its dual space.
  Then we  have  the    contravariant  functor $M \mapsto M^{\sigma} $   \cite{DGK} acting on modules  from  the category ${\mathcal O}$.
  Here $M^{\sigma} =  \bigoplus_{\mu \in \Omega(M)} M_{\mu} ^ *$ is the $\widehat{\g}$--module uniquely determined by
  $$  \langle y  w', w \rangle = \langle w', \sigma (y) w \rangle, \quad y  \in \widehat{\g}, \  w' \in M^{\sigma}, \ w \in M. $$
  It is easy to see that $M$ admits the decomposition
 \bea  M = \bigoplus_{\alpha \in {\C} } M (\alpha), \quad L_{\g} (0) \vert M(\alpha) \equiv \alpha\, \mbox{Id}
\label{wt-dec-l0} \eea
  such that :
  \begin{itemize}
  \item for any $\alpha \in {\C}$ we have $M(\alpha -n) = 0$ for $n \in {\Z}$ sufficiently large;
  \item  for any $\mu \in \Omega(M)$ there exist $\alpha \in {\C}$ such that  $M_{\mu} \subset M(\alpha)$.
  \end{itemize}
  \begin{prop}\label{fd} Assume that  a module $M$ is  in the category $\mathcal O^k$. Then $M$ is in the category $KL^k$ if and  only if $M$ is $\g$-locally finite.
\end{prop}
\begin{proof}
If $M$ is in $KL^k$ then it admits a decomposition as in  \eqref{wt-dec-l0}. Since the spaces $M(\a)$ are $\g$--stable and finite-dimensional, $M$ is $\g$--locally finite.

Let us prove the converse. If $M$ is a highest weight module which is $\g$--locally finite, then clearly all eigenspaces for $L_{\g}(0)$ are finite-dimensional.
Assume now that $M$ is an arbitrary $\g$--locally finite module in the category ${\mathcal O}^k$. Take $\alpha \in {\C}$ such that $M(\alpha) \ne \{0\}$.  Then    from  \cite[Proposition 3.1]{DGK}
 we see  that $M$   has an increasing filtration (possibly infinite)
    \bea  \{0\} = M_0 \subset M_1 \subset  \cdots  \subset M  \label{filt-1} \eea
       such that for every $j \in {\Z}_{>0}$, $M_j / M_{j-1} \cong \widetilde L(\lambda_j) $ is a  highest weight  $V^k(\g)$--module  with highest weight $\lambda_j$,  which is $\g$--locally finite.  Let $h_{\lambda_j}$ denotes the lowest conformal weight of $\widetilde L(\lambda_j)$. Since the factors $M_i/M_{i-1}$ ($i\le j$) of $M_j$ are  highest weight modules, their $L_\g(0)$--eigenspaces are finite-dimensional. This implies that the $L_\g(0)$--eigenspaces of  $M_j$ is  finite-dimensional. By using the properties of the category $\mathcal O$ one sees the following:
 \begin{itemize}
 \item There exists a finite subset $\{ d_1, \cdots, d_s\} \subset {\C}$ such that
 $ \alpha  \in \bigcup_{i=1} ^s (d_i + {\Z}_{\ge 0})$.
 \item For $d \in {\C}$ there exist only finitely many subquotients $\widetilde L(\lambda_j)$ in  (\ref{filt-1}) such that $ h_{\lambda_j} = d$.
 \end{itemize}
 This implies that there is $j_0 \in {\Z}_{>0}$ such that    $\alpha < h_{\lambda_j}$ for $j \ge j_0$. Therefore $M(\alpha) \subset M_{j_0} $. This proves that $M(\alpha)$ is finite-dimensional.
\end{proof}

\begin{rem}\label{rm} We will use several times the following fact, which is a consequence of the previous proposition: for any $k\notin \mathbb Z_{\geq 0}$ and any irreducible highest weight module $L(\l)$ in the category $KL^k$, one has  $\l(\a_0^\vee)\notin  \mathbb Z_{\geq 0}$.\end{rem}
Since $\sigma(L_\g(0))=L_\g(0)$,  if $M$ is in the category $KL^k$, then $M ^{\sigma}$ is also in the category $KL^k$.
  The next result shows that this functor acts on the category   $KL_k$. In the proof we find an explicit  relation of $M^{\sigma}$ with the contragradient modules, defined for ordinary modules for  vertex operator algebras \cite{FHL}.
 \begin{lem} \label{chev}
 \item [(1)]   Assume that  $M$ is  a $V_k(\g)$--module  in  the category $\mathcal O$. Then $M ^{\sigma} $  is also a $V_k(\g)$--module in the category $\mathcal O$.
  \item[(2)] Assume that  $M$ is  a  $V_k(\g)$--module  in  the category $KL_k$. Then $M ^{\sigma} $ is  also in  $KL_k$.
 \end{lem}
 \begin{proof}
Assume  that $M$ is  a $V_k(\g)$--module in the category $\mathcal O$. Take the weight decomposition   $M  = \bigoplus_{\mu \in \Omega(M)} M_{\mu}$, and set  $M^{c}  = \bigoplus_{\mu \in \Omega(M)} M_{\mu} ^*$. By applying the same approach as in the construction of the contragredient module from   \cite[Section 5]{FHL},  we get
 a $V_k(\g)$--module $(M^c, Y_{M^c} (\cdot, z) )$, with  vertex operator map
\bea  \langle Y_{M^c} (v,z)  w', w \rangle =  \langle w ',    Y_M(   e^ { z L_{\g} (1)}  (- z^{-2} ) ^{L_{\g} (0)} v, z)   w\rangle,   \label{vo-map} \eea
  where $w' \in M^c$, $w \in M$.
The  ${\widehat {\g} }$--action on $M ^c$ is uniquely determined by
 $$ \langle  x(n)  w' ,  w \rangle =  - \langle w' , x(-n) w \rangle  \quad (x \in {\g} ). $$
As  a  vector space $M^c = M^{\sigma}$,  but we have different actions of $\widehat{\g}$. (Note that, in  general,  $M^c$ can be outside of the category $\mathcal O$.)

 Take the Lie algebra automorphism  $h  \in Aut (\g)$ such that $$ e_i  \mapsto - f_i  , \ \ f_i   \mapsto   - e_i , \ h_i  \mapsto -h_i  \quad (i = 1, \dots, \ell).$$  Then $h$ can be lifted to an automorphism of $V^k(\g)$. Since the maximal ideal of $V^k(\g)$ is  unique, then it is $h$--invariant, thus  $h$ is also an automorphism of $V_k(\g)$.
  Then we can define a $V_k(\g)$--module $( M^c _h,  Y_{M^c _h} (\cdot, z) )$  where
 $$  M^c _h := M^c, \quad  Y_{M^c _ h } (v, z) =  Y_{M^c } ( h  v, z). $$
On  $M^ c  _h $ we have
 $$ \langle  e_i  (n)  w' ,  w \rangle =   \langle w' , f_i (-n) w \rangle    $$
$$ \langle  f_i (n)  w' ,  w \rangle =   \langle w' , e_i (-n) w \rangle   $$
$$ \langle  h_i   (n)  w' ,  w \rangle =   \langle w' , h_i (-n) w \rangle   $$
where $i=1, \dots, \ell$.
This implies that $M^c _ h = M ^{\sigma}$.
 This proves the assertion (1).

 Assume now that $M$ is in the category $KL_k$. Then all $L_{\g}(0)$--eigenspaces  are finite-dimensional, thus $$ M^{c}  = \bigoplus_{\mu \in \Omega(M)} M_{\mu} ^* = \bigoplus_{\alpha \in {\C} } M(\alpha) ^*.$$
 This implies the   $V_k(\g)$--module $(M^c, Y_{M^c} (\cdot, z) )$ coincides  with    the contragredient  module  \cite{FHL},   realized on the restricted dual space  $\bigoplus_{\alpha \in {\C} } M(\alpha) ^*$, with the vertex operator map (\ref{vo-map}).
  Since  the $L_{\g}(0)$--eigenspaces of $M^c$ are finite-dimensional, we conclude that $M^c$ and $M^{\sigma} = M^c _h$  are $V_k(\g)$--modules in $KL_k$. Claim (2) follows.
 \end{proof}

\section{Constructions of  vertex algebras  with one irreducible module in $KL_k$ via collapsing levels}

By \cite{AKMPP-JA}, if $k$ is a collapsing level, then either   $W_k(\g,\theta) = \C$, $W_k(\g,\theta)=M(1)$, or $W_k(\g,\theta)=V_{k'}(\mathfrak a)$ for a unique 
simple component $\mathfrak a$ of $\g^\natural$. Here the level $k'$ is computed with respect to the invariant bilinear form of $\mathfrak a$ normalized so that the minimal root has squared length $2$. For $\mathfrak a= sl(m|n)$, $m\ge2$, the minimal root is always chosen to be the lowest root of $sl(m)$. For $\mathfrak a=osp(m|n)$ we write $spo(n|m)$ vs. $osp(m|n)$ to specify the choice of the minimal root. In all other cases the minimal root of $\mathfrak a$ is unique.

To simplify notation define $V_{k'}(\g^\natural)$ to be as follows: 
$$V_{k'}(\g^\natural)=\begin{cases}\C &\text{if $W_k(\g,\theta) = \C$; in this case we set $k'=0$;}\\
M(1) &\text{if $W_k(\g,\theta) = M(1)$; in this case we set $k'=1$;}\\
V_{k'}(\mathfrak a)&\text{otherwise.}\end{cases}$$

In Table 5 we summarize all the relevant data.

\begin{table}\centerline{Table 5}
	
\vskip 5pt
\noindent {\sl Values of $k$ and $k'$.}
\vskip 5pt

\centerline{\renewcommand{\arraystretch}{1.5}\begin{tabular}{c|c|c|c}
$\g$&$V_{k'}(\g^\natural)$&$k$&$k'$\\
\hline
$sl(m|n)$, $m\neq n, m>3, m-2\ne n$&$V_{k'}(sl(m-2|n)))$&$\frac{n-m}{2}$&$\frac{n-m+2}{2}$\\ 
\hline
$sl(3|n)$, $n\ne3, n\ne 1, n\ne0$&$V_{k'}(sl(1|n)))$&$\frac{n-3}{2}$&$\frac{1-n}{2}$\\ 
\hline
$sl(3)$&$\C$&$-\frac{3}{2}$&$0$\\ 
\hline
$sl(2|n)$, $n\neq 2,n\ne1, n\ne 0$&$V_{k'}(sl(n)))$&$\frac{n-2}{2}$&$-\frac{n}{2}$\\ 
\hline
$sl(2|1)=spo(2|2)$&$\C$&$-\frac{1}{2}$&$0$\\ 
\hline
$sl(m|n)$, $m\neq n,n+1,n+2, m\ge2$&$M(1)$&$-1$&$1$\\ \hline
$psl(m|m),\,m\ge2$&$\C$& $-1$&$0$\\
\hline
$spo(n|m),m\ne n,n+2,n\ge 4$& $V_{k'}(spo(n-2|m))$ &$\frac{m-n-4}{4}$&$\frac{m-n-2}{4}$\\
\hline
$spo(2|m),m\ge  5$& $V_{k'}(so(m))$ &$\frac{m-6}{4}$&$\frac{4-
m}{2}$\\
\hline
$spo(2|3)$& $V_{k'}(sl(2))$ &$-\frac{3}{4}$&$1$\\
\hline
$spo(2|1)$& $\C$ &$-\frac{5}{4}$&$0$\\
\hline
$spo(n|m),m\ne n+1,n\ge 2$& $\C$ &$-1/2$&$0$\\
\hline
$osp(m|n),m\ne n,m\ne n+8,m\geq 7$&$V_{k'}(osp(m-4|n))$ &$\frac{n-m+4}{2}$&$\frac{8-m+n}{2}$
\\\hline
$osp(m|n),n\ne m,0;4\le m\le 6$&$V_{k'}(osp(m-4|n))$ &$\frac{n-m+4}{2}$&$\frac{m-n-8}{4}$
\\\hline
$osp(m|n),m\ne n+4,n+8;m\geq 4$&$ V_{k'}(sl(2))$ &$-2$ &$\frac{m-n-8}{2}$\\
\hline
$osp(n+8|n),n\geq 0$&$\C$ &$-2$ &$0$\\
\hline
$D(2,1;a)$&$V_{k'}(sl(2))$&$a$ &$-\frac{1+2a}{1+a}$\\
\hline
$D(2,1;a)$&$V_{k'}(sl(2))$&$-a-1$ &$-\frac{1+2a}{a}$\\
\hline
 $F(4)$&$V_{k'}(D(2,1;2))$ &$-1$& $\frac{1}{2}$\\
\hline
 $F(4)$&$\C$ &$-3/2$& $0$\\
\hline
$F(4)$&$V_{k'}(so(7))$&$\frac{2}{3}$ & $-2$\\\hline
$F(4)$&$\C$&$-\frac{2}{3}$ & $0$\\\hline
$E_6$&$V_{k'}(sl(6))$&$-4$ & $-1$\\\hline
$E_6$&$\C$&$-3$ & $0$\\\hline
$E_7$&$V_{k'}(so(12))$&$-6$ & $-2$\\\hline
$E_7$&$\C$&$-4$ & $0$\\\hline
$E_8$&$V_{k'}(E_7)$&$-10$ & $-4$\\\hline
$E_8$&$\C$&$-6$ & $0$\\\hline
$F_4$&$V_{k'}(sp(6))$&$-3$ & $-\frac{1}{2}$\\\hline
$F_4$&$\C$&$-5/2$ & $0$\\\hline
$G_2$&$V_{k'}(sl(2))$&$-\frac{4}{3}$&$1$\\\hline
$G_2$&$\C$&$-\frac{5}{3}$&$0$\\\hline
$G(3)$&$V_{k'}(G_2)$&$\frac{1}{2}$ & $-\frac{5}{3}$\\\hline
$G(3)$&$\C$&$-\frac{3}{4}$ & $0$\\\hline
$G(3)$&$V_{k'}(osp(3|2))$&$-\frac{2}{3}$ & $1$\\\hline
$G(3)$&$\C$&$-\frac{4}{3}$ & $0$
\end{tabular}}

\end{table}

Assume that $k\notin \mathbb Z_{\geq 0}$  and  that:
\begin{itemize}
\item[(1)] $k$  is a collapsing level for $\g$;
\item[(2)] $V_{k'}(\g ^{\natural} )$ is the unique  irreducible $V_{k'}(\g ^{\natural} )$--module in the category $KL_{k'}$.
\end{itemize}
Assume that $L (\widehat\L)$ is an irreducible $V_{k}(\g)$-module in the category $KL_{k}$. Set $\mu= \widehat \L_{|\h}$. By Proposition \ref{fd}   we have $\mu\in P_+$, hence, by \eqref{mumus}, the weight $\mu$ has the form $\mu = \mu ^{\natural} + \ell  \theta$ with $\ell  \in \frac{1}{2} {\Z}_{\ge 0}$, where
$ \mu^{\natural} = \mu _{|\mathfrak h^{\natural}}$.

Since $k \notin {\mathbb Z}_{\ge 0}$, by Theorem \ref{A}, $H_{\theta} (L (\widehat\L))$ is a non-trivial irreducible module for $W_k(\g,\theta)$. Since $L (\widehat\L)$ is a quotient of the Verma module $M(\widehat \L)$, then, by exactness of $H_\theta$,  $H_{\theta} (L (\widehat\L))$ is the quotient of a Verma module for $W_k(\g,\theta)=V_{k'}(\g ^{\natural} )$ hence it is an irreducible highest weight module. By \cite[(6.14)]{KW2} its highest weight as $\mathcal V_{k}(\g ^{\natural} )$-module is $\widehat \L^\natural$ with $\widehat \L^\natural(K)=k'$ and $\widehat \L^\natural_{|\h^\natural}=\mu^\natural$. Therefore
$$H_{\theta} (L (\widehat\Lambda)) =L_{\g ^{\natural} } (\widehat\L^\natural). $$
In particular  $H_{\theta} (L (\widehat\Lambda ))$ is in the category $KL_{k'}$.

Moreover, under the identification of the centralizer $\g^f$ of $f$ in $\g$ with $\g_0\oplus \g_{1/2}$ via $ad(f)$ (see Example 6.2 of \cite{KW2}),
 we get that $x$ acts on $H_{\theta} (L (\widehat\Lambda ))$ via $J^{\{f\}}_0$, and $J^{\{f\}}$ is the conformal vector of $W(k,\theta)$ (see the proof of Theorem 5.1 of \cite{KW2}). Since the level is collapsing we know, by Proposition 4.1 of \cite{AKMPP-JA}, that the conformal vector  of $W_k(\g,\theta)$ coincides with the Segal-Sugawara vector conformal  $\omega_{\g ^{\natural}}$  of  $V_{k'}(\g ^{\natural} )$ hence, by (6.14) of \cite{KW2} again, we obtain that the
 $( \omega_{\g ^{\natural}} )_0$  acts  on the lowest component of $H_{\theta} (L (\widehat\L))$  by $cI$ with
\bea c=\frac{ (\mu + 2 \rho, \mu)} {2 (k+h^{\vee})} - \mu (x). \label{lowest-conf-weight-W} \eea

Now condition (2) implies that
$\mu^{\natural} = 0$,
so $\mu = \ell \theta$ and
$$ \frac{ (\mu + 2 \rho, \mu)} {2 (k+h^{\vee})} - \mu (x) = \frac{ (\ell \theta  + 2 \rho, \ell \theta )} {2 (k+h^{\vee})} -\ell = 0. $$
By using formula \eqref{rhotheta},  we get
\begin{equation}\label{actionw}
 \frac{ 2 \ell ^2 +(2 h^{\vee} -2)\ell }{2 (k+ h^{\vee})} - \ell = \frac{  \ell ^2   -(k+1)\ell }{ k+ h^{\vee}}= 0.
 \end{equation}
\begin{itemize}
\item Consider first the case $k = - h^{\vee} /2+1$ (this holds for $\g = D_{2 n}$, $n \ge 2$ and $\g =osp( n + 4 m + 8  | n)$, $n \ge 0$).
Then  \eqref{actionw} gives that
\bea \frac{   2 \ell ^2 +  (h^{\vee} -4 )\ell  }{ h^{\vee} +2 } = 0. \label{uv-7} \eea
We get $\ell =0$ or  $ 2 \ell + h^{\vee} -4 = 0$.

\item Next we consider  the case $k = - h^{\vee} /6 -1$. We get
\bea   \frac{  6  \ell ^2 + h^{\vee} \ell  }{5  h^{\vee} -6 } = 0.  \label{uv-8} \eea
We conclude that $\ell = 0$ or $  \ell  = - \frac{ h^{\vee} }{6}$.
\end{itemize}
By using the  above analysis and properties of Hamiltonian reduction, we get the following lemma,
which  extends a  result of \cite{AM} for Lie algebras to the super case.
\begin{lem} \label{lemma-pomoc-1} Assume that $k=-\frac{h^\vee}{6}-1$ and $\g$ is one of the Lie algebras of exceptional Deligne's series $A_2$, $G_2$, $D_4$, $F_4$, $E_6$, $E_7$, $E_8$, or $\g=psl(m|m)$ ($m\ge2$), $osp(n+8|n)$ ($n\ge2$), $spo(2|1)$, $F(4)$, $G(3)$ (for both choices of $\theta$).

Assume that  $L(\lambda)$  is a $V_k(\g)$--module in the category $\mathcal O$. Then one of the following condition holds:
\begin{itemize}
\item[(1)] $\lambda (\alpha_0 ^{\vee} ) \in {\Z}_{\ge 0}$;
\item [(2)] $\bar \lambda$  is either  $0$ or $  \frac{-h^{\vee}}{6} \theta $, where $\bar \lambda$ is the restriction of $\lambda$ to  $\h$.
\end{itemize}
\end{lem}
\begin{proof}
By Theorem \ref{A}, if  $L(\lambda)$  is a  $V_k(\g)$--module for which  $\lambda (\alpha_0 ^{\vee} ) \notin {\Z}_{\ge 0}$, then $ H_{\theta} (L(\lambda)) $  is an irreducible  $W_k(\g, \theta) =H_{\theta} (V_k(\g))$--module.  The conditions on $\g$ exactly correspond to the cases when $W_k(\g, \theta)$ is one-dimensional (cf. \cite{AKMPP-JA}, \cite{AM}), so the discussion that precedes the Lemma and  relation (\ref{uv-8}) imply that
$\bar \lambda$ is as in (2).
 \end{proof}

Lemma  \ref{lemma-pomoc-1}  implies:
\begin{thm} \label{thm-classification-unique}
Assume  that the level $k$ and the Lie superalgebra $\g$ satisfy one of the following conditions:
\begin{itemize}
\item[(1)]   $k=-\frac{h^\vee}{6}-1$ and $\g$ is one of the Lie algebras of exceptional Deligne's series $A_2$, $G_2$, $D_4$, $F_4$, $E_6$, $E_7$, $E_8$, or $\g=psl(m|m)$ ($m\ge2$), $osp(n+8|n)$ ($n\ge2$), $spo(2|1)$, $F(4)$, $G(3)$ (for both choices of $\theta$);
\item[(2)]  $k =-h^{\vee} / 2 + 1$ and $\g =osp(n + 4 m + 8 \vert n)$, $n \ge 2, m \ge 0 $.
\item[(3)] $k =-h^{\vee} / 2 + 1$ and $\g = D_{2 m}$, $m \ge 2$.
\item[(4)] $k = -10$ and $\g = E_8$.
\end{itemize}
Then $V_k(\g)$ is the unique irreducible $V_k(\g)$--module in the category $KL_k$.
\end{thm}
\begin{proof}
If the   Lie superalgebra $\g$ is as in (1), then  Lemma  \ref{lemma-pomoc-1}  and Remark \ref{rm} imply that $\bar \lambda$  is either  $0$ or $  \frac{-h^{\vee}}{6} \theta $. Since  in all cases in (1) we have that  $h ^{\vee} \in {\Z}_{\ge 0}$, one obtains that   the irreducible highest weight  $\g$--module with highest weight  $\bar \lambda = \frac{-h^{\vee}}{6} \theta$  cannot be  finite-dimensional. Therefore $L(\lambda)$ can  not be  a module in $KL_k$. This proves that  $\bar \lambda = 0$ and therefore $V_k(\g)$ is the unique irreducible $V_k(\g)$--module in the category $KL_k$.

 Let us consider   the case  $\g =osp( n + 4m + 8 \vert  n)$.  Then for every $m \in {\Z}_{\ge 0}$ we have:
  \bea
  && h^{\vee} =  4m + 6, \label{cond-0} \\
  && k =-h^{\vee} / 2 + 1 = - 2 (m+1), \label{cond-1}  \\
  && 2 \ell + h^{\vee} - 4  \ne  0 \quad \forall \ell \in \frac{1}{2} {\Z}_{\ge 0}. \label{cond-2}
  \eea

 We prove the claim by induction.  In the case $m= 0$,  the claim was proved in (1). Assume now that the claim holds for $\g' =osp( n + 4 (m -1) + 8, n)$, and
$k' =-2m $.

 By Theorem \ref{TT}, $k = - 2 (m+1)$ is a collapsing level and 
$W_{k} (\g, \theta) = V_{k'} (\g')$.

 By inductive assumption  $V_{k'} (\g')$ is the unique irreducible  $V_{k'} (\g')$ in the category $KL_{k'}$.  By applying (\ref{uv-7})  and (\ref{cond-2}) we get that $\ell = 0$ and therefore   $V_{k}(\g)$ is the unique irreducible $V_k(\g)$--module in the category $KL_k$.  The assertion now follows by induction on $m$.

(3) is a special case of  (2), by taking $n=0$.

(4) follows from the fact that $H_\theta(V_{-10}(E_8))= V_{-4}(E_7)$ and case (1) by applying  formula (\ref{actionw}).
\end{proof}
 \begin{rem}
 Theorem \ref{thm-classification-unique} can be also proved by non-cohomological methods, using explicit formulas for singular vectors and Zhu algebra theory. As an illustration, we shall present in Theorem   \ref{class-d-direct} a direct proof in the case of  $D_{2n}$ at level $k =-h^{\vee} / 2 +1$.
 \end{rem}
 In the following sections we shall study some other applications
 of collapsing levels. We shall restrict our analysis to the case of Lie algebras.    In what follows we let $\omega_1, \ldots , \omega_{n}$ be the fundamental weights for ${\mathfrak g}$ and
 $\Lambda_0, \ldots , \Lambda_{n}$ the
fundamental weights for $\widehat{\mathfrak g}$.

\section{ On complete reducibility in the category  $KL_k$ }

 In this Section we prove complete reducibility results in the category  $KL_k$ when $\g$ is a Lie algebra. We start with a preliminary result, which also holds in the super setting.
\begin{lem} \label{hw} Assume that the Lie superalgebra $\g$ and level $k$ satisfy  the conditions of  Theorem \ref{thm-classification-unique}.  Assume that $M$ is a highest weight $V_k(\g)$--module from the category $KL_k$. Then $M$ is irreducible.
\end{lem}
\begin{proof}
 By using the classification of irreducible modules from Theorem \ref{thm-classification-unique} we know that the highest weight of $M$ is necessary $k \Lambda_0$, and therefore $M$ is a ${\Z}_{\ge 0}$--graded with respect to  $L_{\g} (0)$. Denote   a  highest weight vector by $w_{k \Lambda_0}$.  We have that
 $$ L_{\g}(0) v = 0 \quad \iff \quad v = \nu w_{k \Lambda_0} \quad (\nu \in {\C}). $$
 Assume that $M$ is not irreducible.
  Then it contains a   non-zero graded  submodule $N \ne   M $ with respect to $L_{\g}(0)$:
$$ N = \bigoplus_{n \in  {\Z}_{\ge 0}  }  N (n+n_0), \quad L_{\g} (0)_{\vert N(r)}  = r \mbox{Id}, \quad N(n_0) \ne 0.$$
Since $N \ne  M $  , we have that $n_0 >0$, otherwise $ w_{k \Lambda_0} \in M$.

We can consider $N(n_0)$ as a finite-dimensional module for $\g$ and for the Zhu algebra $A (V_k(\g))$. Note that Theorem  \ref{thm-classification-unique}  and  Proposition \ref{zhu-int}  imply that any irreducible finite-dimensional $A (V_k(\g))$--module is trivial.
Since the Casimir element $C_{\g}$  of $\g$ acts  on $N(n_0)$ as the  non-zero constant    $ 2 (k+h ^{\vee}) n_0 $, we conclude that $C_{\g}$ acts  by the same constant on any irreducible $\g$--subquotient of $N(n_0)$. But  any irreducible subquotient of $N(n_0)$ is  an irreducible finite--dimensional $A(V_k(\g))$--module, and therefore it is trivial.
This implies that $C_{\g}$ acts non-trivially on a trivial $\g$--module, a contradiction.
\end{proof}
The following Lemma is a consequence of \cite[Theorem 0.1]{GK}.
\begin{lem} \cite{GK} \label{gk} Assume that $\g$ is a simple  Lie algebra and $k$ is a rational   number,   $k > - h^{\vee}$. Then,  in the category of $V_k(\g)$--modules, 
we have:
$\mbox{Ext}^1 (V_k(\g), V_k(\g)) = ( 0). $
\end{lem}
\begin{thm} Assume that  $\g$   is a simple   Lie algebra and that the level $k$ satisfies   the conditions of  Theorem \ref{thm-classification-unique}.
Then  any $V_k(\g)$--module $M$ from the category $KL_k$ is  completely  reducible.
\end{thm}
\begin{proof}
Since $M$ is in $KL_k$ we have that any irreducible subquotient of $M$ is isomorphic to $V_k(\g)$. $M$ has finite length. This implies that $M$ is ${\Z}_{\ge 0}$--graded:
$$ M = \bigoplus_{n \in  {\Z}_{\ge 0}  } M (n), \quad L_{\g} (0)_{\vert M(r)}  = r \mbox{Id}.$$
Assume that $M(0) = \span_{\C} \{ w_1, \dots, w_s\}$. Then by Lemma \ref{hw} we have that $V_k(\g) w_i \cong V_k(\g)$ for every $i=1, \dots, s$. Now using Lemma \ref{gk} we get
$M \cong \oplus V_{k} (\g) w_i$ and therefore $M$ is  completely reducible.
\end{proof}
\begin{rem} We expect that the previous theorem holds in the case when $\g$ is the Lie superalgebra  from   Theorem \ref{thm-classification-unique}.
We shall study this case in \cite{AKMPP-work-in-progress}. \end{rem}
We shall now  prove much more general result on complete reducibility in $KL_k$.
\begin{thm} \label{general-complete-reducibility}
  Assume that   level  $k \in {\Q}$, $k> -h ^{\vee}$,  and  the simple Lie algebra $\g$ satisfy  the following property:
  \bea  \mbox{ Every highest weight} \  V_k(\g) \mbox{--module  in} \  KL_k  \   \mbox{is irreducible}. \label{asum-hwm} \eea
Then  the category $KL_k$ is semi-simple.
\end{thm}
\begin{proof}
We shall present a sketch of the proof and omit some standard representation theoretic arguments which can be found in \cite{DGK} and \cite{GK}.
\begin{itemize}
 \item   Since   every irreducible  $V_k(\g)$-module  in  $KL_k$  is isomorphic to $L(\lambda)$ for certain rational, non-critical weight $\lambda$, then \cite[Theorem 0.1]{GK} implies that
$ \mbox{Ext}^1 (L(\lambda) , L(\lambda)  ) = ( 0 ) $  in the category $KL_k$.
\item
We prove that in the category $KL_{k}$ we have
\begin{equation} Ext^1 (L_1, L_2 ) =  (0 )  \label{ext2} \end{equation}
   for any two irreducible modules $L_1$ and $L_2$ from  $KL_{k}$.

It remains  to consider the case  $L_1 \ne L_2$.  Take an   exact sequence in $KL_k$:
   $$ 0 \rightarrow L(\lambda_1) \rightarrow M \rightarrow L(\lambda_2) \rightarrow 0, $$
   where $\lambda_1 \ne \lambda_2$.
   Then $M$ contains a singular vector $w_{\lambda_1}$ of highest weight $\lambda_1$  and a subsingular vector $w_{\lambda_2}$ of weight $\lambda_2$ and $w_{\lambda_1}$ generates  a submodule isomorphic to $L(\lambda_1)$.
   Consider the case $\lambda_1 -\lambda_2 \notin Q_+$. Then $\lambda_2$ is a maximal element of  the set $\Omega(M)$ of  weights of $M$,  and therefore the  subsingular vector $w_{\lambda_2}$  in $M$ of weight $\lambda_2$ is  a singular vector. By  (\ref{asum-hwm}),  it generates an irreducible module isomorphic to $L(\lambda_2)$ and we conclude that $M\cong L(\lambda_1) \oplus L(\lambda_2)$.

   If $\lambda_1 -\lambda_2 \in Q_+$ we can use the  contravariant  functor $M \mapsto M^{\sigma} $     and get an exact sequence
     $$ 0 \rightarrow L(\lambda_2) \rightarrow M^{\sigma}  \rightarrow L(\lambda_1) \rightarrow 0. $$
     Since $ M^{\sigma} $ is again a $V_k(\g)$--module in $KL_k$  (cf. Lemma \ref{chev})  by the first case we have that $M^{\sigma} =  L(\lambda_1) \oplus L(\lambda_2)$. 
     This  implies that
     $$M=L(\lambda_1)^{\sigma}  \oplus L(\lambda_2)^{\sigma} =L(\lambda_1)  \oplus L(\lambda_2). $$

  \item
    Assume now  that $M$ is a  finitely generated  module from $KL_{k}$.  Then    from  \cite[Proposition 3.1]{DGK}
 we see  that $M$   has an increasing filtration 
 \begin{equation}\label{cs}(0) = M_0 \subseteq M_1 \subseteq  \cdots   \end{equation}
 such that 
   \begin{enumerate} \item for every $j \in {\Z}_{>0}$,  $M_j / M_{j-1}$ is an highest weight  module in category $\mathcal O$;
   \item for any weight $\l$ of $M$, there exists $r$ such that $(M/M_r)_\l=0$.
   \end{enumerate}
Since $M$ is finitely generated as $\widehat{\g}$--module, we can assume that its generators are weight vectors of
weights say $\mu_1, ...\mu_p$. Since they are a finite number there
certainly exists  $t$ such that $(M/M_t)_{\mu_i}=0$ for all $i=1,..,p$.
Hence the filtration \eqref{cs} is finite and stops at  $M=M_t$. Since  $M$ is in category $KL_k$, we have that the factors of \eqref{cs} are in category $KL_k$. Hence, by our assumption, they   
are  irreducible.  Therefore \eqref{cs} is  a composition series of finite length.
        Using assumption  (\ref{asum-hwm}),  relation (\ref{ext2}) and induction on $t$  we get that
    $$M \cong \bigoplus _{j=1} ^t L(\lambda_j). $$
 \item
 Finally, we shall consider the case when $M$ is  not finitely generated. Since $M$ is in $KL_k$, it is countably  generated. So
 $M = \cup_{n=1} ^{\infty} M^{(n)}$ such that each $M^{(n)}$ is  finitely generated $V_k(\g)$--module.  By   previous case  $M^{(n)}$ is completely reducible, so:
\bea  M^{(n)} = \bigoplus_{i = 1} ^{n_i}  L(\lambda_{i,n}). \label{dec-n}\eea
Therefore $M$ is a sum of irreducible modules from $KL_k$  and by using classical algebraic arguments one can see that $M$ is a direct sum of countably many  irreducible modules from $KL_k$ appearing in decompositions (\ref{dec-n}).
   \end{itemize}
 The claim  follows.
 \end{proof}
In order to apply Theorem \ref{general-complete-reducibility}, the basic step is to check relation (\ref{asum-hwm}). We have the following method.
  \begin{lem}  \label{kriterij} Let $k \in {\Q} \setminus {\Z_{\ge 0}}$.
  Assume that $H_{\theta} (U)$ is an irreducible, non-zero $W_k(\g, \theta) = H_{\theta} (V_k(\g))$--module  for every non-zero highest weight  $V_k(\g)$--module  $U$ from  the category  $KL_k$. Then every highest weight $V_k(\g)$--module in $KL_k$  is irreducible.
  \end{lem}
  \begin{proof}
  Assume that $M$ is a highest weight $V_k(\g)$--module in $KL_k$. Then $H_{\theta} (M)$ is an  irreducible $H_{\theta} (V_k(\g))$--module. If $M$ is not irreducible, then it contains a highest weight submodule $U$ such that
  $ \{0\}\subsetneqq U  \subsetneqq M$. Modules $U$ and $M / U$ are again highest weight modules  in $KL_k$.
By  the assumption of the Lemma we have that $H_{\theta} (U)$ is a non-trivial submodule of $H_{\theta} (M)$. Irreducibility of $H_{\theta} (M)$ implies that $H_{\theta} (U) = H_{\theta} (M)$, and therefore $H_{\theta} (M / U) = \{0\}$, a contradiction.
  \end{proof}
  \begin{thm}  \label{rational}
Assume that $\g$ is a  simple Lie algebra and  $k \in {\C}\setminus {\Z}_{\ge 0}$ such that  $W_k(\g, \theta)$ is rational. Then $KL_k$ is a semi-simple category.
\end{thm}
 \begin{proof}
 Assume that $\widetilde L(\lambda)$ is a  highest weight $V_{k}(\g)$--module in $KL_k$. Clearly  $\lambda (\alpha_0^{\vee}) \notin {\Z}_{\ge 0}$ and by Theorem \ref{A} $H_{\theta} ( \widetilde L(\lambda) ) \ne (0)$. Since  $H_{\theta} ( \widetilde L(\lambda) )$ is non-zero highest weight module for the  rational vertex algebra $W_k(\g, \theta)$, we conclude that $H_{\theta} ( \widetilde L(\lambda) )$ is irreducible. Now assertion follows from Theorem \ref{general-complete-reducibility}  and Lemma \ref{kriterij}.
 \end{proof}

 \begin{rem}
 The previous theorem proves that the category $KL_k$ is semisimple in the following (non-admissible) cases:
 \begin{itemize}
 \item $\g=  D_4, E_6, E_7, E_8$ and $k = - \frac{h ^{\vee}}{6}$ using   results from \cite{Kaw}.
 \end{itemize}
 \end{rem}
Moreover, using Theorem \ref{general-complete-reducibility}  and Lemma \ref{kriterij} we  can prove the semi-simplicity of $KL_k$  for all collapsing levels not accounted   by
Theorem \ref{thm-classification-unique-introd}. We list here only non-admissible levels, since in admissible case $KL_k$ is semi-simple by \cite{Araduke}.
 \begin{thm} \label{cases-semi-simple}
 The category $KL_k$ is semisimple in the following cases:
\begin{itemize}
\item[(1)]  $\g = D_{\ell}$, $\ell \ge 3$ and $k=-2$;
\item[(2)] $\g = B_{\ell}$, $\ell \ge 2$ and $k = -2$;
\item[(3)]  $\g = A_{\ell}$, $\ell \ge 2$ and $k =-1$;
\item [(4)] $\g = A_{2\ell-1}$, $\ell \ge 2$, $k =-\ell$;
\item[(5)]  $\g = D_{2 \ell-1}$, $\ell \ge 3$ and $k=-2\ell +3$;
\item[(6)] $\g = C_{\ell}$, $k = -1 -\ell /2$;
\item[(7)] $\g = E_{6}$, $k = -4$;
\item[(8)] $\g = E_{7}$, $k = -6$;
\item[(9)] $\g = F_{4}$, $k = -3$.
\end{itemize}
\end{thm}
\begin{proof} 

 We will give a proof of relations (1) and (2) in Corollaries \ref{semi-simple-d} and \ref{semi-simple-b}, respectively. Case (1) for $\ell\ne 3$ will follow from Theorem \ref{rational}. Note also that case (1) for $\ell=3$ is a special case of case (4), and that case (2) for $\ell=2$ is a special case of (6).  
 The proof in cases (3) -- (6)  is similar, and it uses  the classification of irreducible modules from \cite{AP08}, \cite{AP}, \cite{AM1} and the  results on collapsing levels \cite{AKMPP-JA}. Cases (7) -- (9) are reduced to cases we have already treated. Here are some details.

Case (3):
\begin{itemize}
\item \cite{AM1}, \cite{AKMPP-JA}  $H_{\theta} (V_{-1} (A_{\ell}) )$ is isomorphic to the Heisenberg vertex algebra $M(1)$ of central charge $c=1$
\item By using the fact that every  highest weight $M(1)$--module  is irreducible,  we see  that if $U$ is a highest weight $V_{-1} (A_{\ell})$--module  in $KL_{-1}$, then $H_{\theta} (U)$ is a non-trivial  irreducible $M(1)$--module.
\end{itemize}

 Case (4):
\begin{itemize}
\item \cite{AM1}, \cite{AKMPP-JA} $H_{\theta} (V_{-\ell } (A_{2 \ell-1}))= V_{-\ell+1 } (A_{2 \ell-3}) $.
\item For $\ell = 2$, we have that every  highest weight $V_{-\ell+1 } (A_{2 \ell-3})  = V_{-1} (sl(2))$--module $\widetilde L (\lambda)$ in $KL_{-1}$  with  highest weight $ \lambda= -(1+j) \Lambda_0 +  j \Lambda_1$, $j \in {\Z}_{\ge 0}$, is irreducible.
\item By induction, we see that for every highest weight $V_{-\ell } (A_{2\ell-1}) $--module  $U$ in $KL_{-\ell }$, $H_{\theta} (U)$ is a non-trivial  irreducible $V_{-\ell+1 } (A_{2 \ell-3}) $--module.
\end{itemize}

Case (5)
\begin{itemize}
\item $H_{\theta} (V_{-2\ell +3} (D_{2\ell-1})) \cong V_{-2\ell +5} (D_{2\ell-3}) $.
 \item By induction we see that for  or every highest weight $V_{-2\ell +3} (D_{2\ell-1})$--module  $U$  in $KL_{-2 \ell +3 }$ , $H_{\theta} (U)$ is a non-trivial  irreducible  $V_{-2\ell +5} (D_{2\ell-3}) $--module.
\end{itemize}

Case (6)
\begin{itemize}
\item   $H_{\theta} ( V_{-1- \ell /2 } (C_{\ell}  )   ) \cong  V_{-1/2  - \ell/ 2} (C_{\ell-1}) $.
\item For $\ell = 2$, we have that every  highest weight $V_{-1/2  - \ell/ 2} (C_{\ell-1})   = V_{-3/2} (sl(2))$--module in $KL_{-3/2}$ is irreducible.
\item  By induction, we see that for   every highest weight $V_{-1- \ell /2 } (C_{\ell}) $--module  $U$  in $KL_{-1 - \ell /2  }$, $H_{\theta} (U)$ is a non-trivial  irreducible $V_{-1/2  - \ell/ 2} (C_{\ell-1})$--module.
\end{itemize}
The proof follows by  applying Theorem \ref{general-complete-reducibility}  and Lemma \ref{kriterij}.
\vskip10pt
Cases (7) -- (8)
\vskip10pt
We have
$$H_\theta(V_{-4}(E_6))= V_{-1}(A_3),\quad
H_\theta(V_{-6}(E_7))= V_{-2}(D_6),\quad
$$
and these cases are settled in (3) and  Theorem \ref{thm-classification-unique-introd} (3) respectively. Case (9) follows from the fact that  $H_\theta(V_{-3}(F_4))$ is isomorphic to the admissible affine vertex algebra $V_{-\frac{1}{2}}(C_3)$ which is semisimple in $KL_{-1/2}$ (cf. \cite{A-1994}).\end{proof}
\begin{rem}
 The problem of complete-reducibility of modules in  $KL_k$ when $\g$ is a Lie superalgebra will be also studied in \cite{AKMPP-work-in-progress}. An important tool in the description of the category  $KL_k$ will be the conformal embedding of $\widetilde V_{k} (\g_0)$ to $V_k(\g)$ where $\g_0$ is the even part of $\g$.
\end{rem}
Note that in the category $\mathcal O$ we can have indecomposable $V_k(\g)$--modules in some cases listed in Theorem \ref{cases-semi-simple}. See \cite[Remark 5.8]{AP08} for one example.

\section{ The vertex algebra $V^{-2} (D_{\ell} ) $ and its quotients  }

In this section  we exploit Hamiltonian reduction and the results on conformal embeddings from \cite{AKMPP-JA} to
 investigate the quotients of the vertex algebra $V^{-2} (D_{\ell} ) $. In particular we are interested in a  non-simple quotient $\mathcal V_{-2} (D_{\ell})$ which appears in  the analysis of certain dual pairs (see \cite{AKMPP-new}) as well as in the simple quotient $V_{-2} (D_{\ell} )$.
We will show that the vertex algebra $\mathcal V_{-2} (D_{\ell})$ has infinitely many irreducible modules in the category $KL^{-2}$,
while by \cite{AM}, $V_{-2} (D_{\ell} )$ has finitely many irreducible modules in $KL_{-2}$. Recall that $-2$ is a collapsing level
for $D_{\ell}$ \cite{AKMPP-JA}.

Consider the vector
\begin{equation} \label{sing-D4-1}
w_1 := (e_{\epsilon_1 + \epsilon_2}(-1) e_{\epsilon_3 + \epsilon_4}(-1)
- e_{\epsilon_1 + \epsilon_3}(-1) e_{\epsilon_2 + \epsilon_4}(-1)
 + e_{\epsilon_1 + \epsilon_4}(-1)
e_{\epsilon_2 + \epsilon_3}(-1) ) {\vac }  .
\end{equation}
It is a singular vector in
$V^{-2}(D_{\ell})$ (cf. \cite{AM}). Note that this vector is contained in the subalgebra
$V^{-2}(D_4)$ of $V^{-2}(D_{\ell})$.

By using the explicit  expression    for singular vectors $v_n$ in $V^ { n - \ell +1 } ( D_{ \ell} )  $  (see   \eqref{singvect-D-old}),  we have that
\begin{equation}\label{w2} w_2 :=v_{\ell -3} = \Big(\sum _{i=2}^{ \ell} e_{\epsilon_1 - \epsilon_i}(-1)
e_{\epsilon_1 + \epsilon_i}(-1)\Big) ^{\ell -3} {\bf 1} \end{equation}
is a singular vector in $V^ { -2 } ( D_{ \ell} )  $.

For $\ell = 4$ we also have  a third   singular vector (cf. \cite{P})
$$w_3 := (e_{\epsilon_1 + \epsilon_2}(-1) e_{\epsilon_3 - \epsilon_4}(-1)
- e_{\epsilon_1 + \epsilon_3}(-1) e_{\epsilon_2 - \epsilon_4}(-1)
 + e_{\epsilon_1 - \epsilon_4}(-1)
e_{\epsilon_2 + \epsilon_3}(-1) ) {\bf 1}.$$

\subsection{ The vertex algebra  $\mathcal V_{-2} (D_{\ell} )$ for $\ell \ge 4$  }

Define the  vertex algebra
\begin{equation}\label{v-2d}\mathcal  V_{-2} (D_{\ell}  ) = V^{-2} (D_{\ell } )\big/  J_{\ell},\end{equation}
where
\bea
 J_{\ell} =  {\langle w_1, w_3  \rangle}  \quad (\ell = 4),   \quad
 J_{\ell} =   {\langle w_1 \rangle} \quad (\ell \ge 5).\nonumber \eea
The following proposition is essentially proven in  \cite{AKMPP-new}.
\begin{prop}
 \item[(1)] There is a non-trivial vertex algebra homomorphism $\overline {\Phi} : \mathcal V_{-2} (D_{\ell} ) \rightarrow M_{2\ell}$ where $M_{2\ell}$ the Weyl vertex algebra of rank $\ell$.
 \item[(2)] $\mathcal  V_{-2} (D_{\ell} )$ is not simple, and  $L( (-2-t) \Lambda_0 + t \Lambda_1)$, $t \in {\Z_{\ge 0}}$ are
$\mathcal V_{-2} (D_{\ell})$--modules.
\end{prop}
\begin{proof}
The homomorphism $\Phi: V^{-2} (D_{\ell}) \rightarrow M_{2 \ell}$ was constructed in \cite[Section 7]{AKMPP-new}.  By direct calculation one proves that $\Phi( w_1) =  0$  for $\ell \ge 4$ and $\Phi(w_3) = 0$ for $\ell = 4$. Finally
  \cite[Lemma 7.1]{AKMPP-new}  implies that $L( (-2-t) \Lambda_0 + t \Lambda_1)$, $t \in {\Z_{\ge 0}}$ are
$\mathcal V_{-2} (D_{\ell})$--modules. Since the simple vertex algebra $V_{-2} (D_{\ell})$ has only finitely many irreducible modules in the category $\mathcal{O}$ \cite{AM}, we have that  $\mathcal V_{-2} (D_{\ell})$  is not simple.
\end{proof}
 Next, we exploit the fact that  in the case $\g= D_{\ell}$, $k=-2$ is a collapsing level, i.e., in the affine $W$-algebra $ W^k(\g, \theta)$, all generators $G^{\{ u \}}$ at conformal weight $3/2$, $u \in {\g}_{-1/2}$, belong to the maximal ideal  (see \cite{AKMPP-JA} for details). This implies that there exists a non-trivial ideal $I$ in $V^{-2} (\g)$  such that $ G^{\{ u \}}\in H_{\theta} (I)$ for all  $u \in {\g}_{-1/2}$.

Note also that   $\g ^{\natural} = A_1 \oplus D_{\ell -2}$, so
we have that $V^{\ell -4} (A_1) \otimes V^{0} (D_{\ell -2} )$ is a subalgebra of
$ W^{-2}(D_{\ell}, \theta)$.
 In the case $\ell = 4$ we identify $D_2$ with $A_1 \oplus A_1$.
\begin{lem} \label{collapsing-d}
We have
\begin{itemize}
\item  $ x_{ (-1)}  {\bf 1} \in  H_{\theta} (J_{\ell} )$ for all $x \in D_{\ell-2} \subset {\g}  ^{\natural} $,
\item  $ G^{\{ u \}}\in H_{\theta} (J_{\ell} )$ for all $u \in {\g}_{-1/2}$.
\end{itemize}
\end{lem}
\begin{proof}
Assume that $\ell \ge 5$.
Since $w_1$ is a singular vector in $V^{-2} (D_{\ell} )$, the ideal
$J_{\ell} $  is a highest weight module  of highest weight $\lambda= -2 \Lambda _0 + \epsilon_1 + \epsilon_2 + \epsilon_3 + \epsilon_4$.
Now, the Main Theorem from \cite{Araduke} implies that $H_{\theta}(J_\ell)$ is a non-trivial highest weight module. By formula \cite[(6.14)]{KW2} the highest weight is $(0 , \omega_2)$ and, by \eqref{lowest-conf-weight-W}, the conformal weight of its highest weight vector is $1$.  Up to a non-zero constant, there is only one vector in $W^{-2}(D_{\ell}, \theta)=V^{\ell -4} (A_1) \otimes V^{0} (D_{\ell -2} )$ that has these properties, namely %
$J^{ \{ e_{\epsilon_3 + \epsilon_4} \} }_{(-1)}  {\bf 1}, $
and therefore   $H_{\theta}(J_{\ell} )$  contains all generators of  $V^{0} (D_{\ell -2} )$.

In the case $\ell = 4$,  $w_1$ and  $w_3$ generate   submodules $N_1$ and  $N_3$ of highest weights  $\lambda_1= -2 \Lambda _0 + \epsilon_1 + \epsilon_2 + \epsilon_3 + \epsilon_4$ ,  $ \lambda_3= -2 \Lambda _0 + \epsilon_1 + \epsilon_2 + \epsilon_3 - \epsilon_4$, respectively. Applying the same arguments as above we get that
$  J^{ \{ e_{\epsilon_3 \pm  \epsilon_4} \} }_{(-1)}  {\bf 1} \in H_{\theta} (I), $
which implies that $H_{\theta}(J_{\ell} )$  contains all generators of  $V^{0} (D_{2} ) = V^{0} (A_1) \otimes  V^{0} (A_1) $.

Now, claim  follows by applying the action of generators of  $V^{0} (D_{\ell -2} )$  to  $ G^{\{ u \}} $ (see \cite{AKMPP-JA}).
\end{proof}
\begin{prop} \label{clas-2-intermediate}
We have
\item[(1)] $H_{\theta} ( \mathcal V_{-2} (D_{\ell}) ) = V^{\ell -4}(A_1)$.
\item[(2)] $ H_{\theta}( L( (-2-t) \Lambda_0 + t \Lambda_1) ) \cong L_{A_1}( (\ell-4-t) \Lambda_0 + t \Lambda_1 ), \ t \in {\Z_{\ge 0}}.$
\item[(3)] The set
$\{ L( (-2-t) \Lambda_0 + t \Lambda_1) \ \vert \ t \in {\Z_{\ge 0}} \}$
provides a complete list of irreducible  $\mathcal V_{-2} (D_{\ell})$--modules from the category $KL^ {-2}$.
\end{prop}
\begin{proof}
By Lemma \ref{collapsing-d} we see that the vertex algebra  $H_{\theta} ( \mathcal V_{-2} (D_{\ell}) ) $ is generated only by $x_{(-1)} {\bf 1}$, $x \in A_1 \subset D_\ell^{\natural}$.
So there are only two possibilities: either   $H_{\theta} ( \mathcal V_{-2} (D_{\ell}) ) =  V^{\ell -4}(A_1) $ or $H_{\theta} ( \mathcal V_{-2} (D_{\ell}) ) =  V_{\ell -4}(A_1) $.
Moreover,  for every $t \in {\Z}_{\ge 0}$, $ H_{\theta}( L( (-2-t) \Lambda_0 + t \Lambda_1) ) $ must be the irreducible  $H_{\theta} ( \mathcal V_{-2} (D_{\ell}) ) $--module with highest weight $t \omega_1$ with respect to $A_1$. So  $ H_{\theta}( L( (-2-t) \Lambda_0 + t \Lambda_1) ) \cong L_{A_1}( (\ell-4-t) \Lambda_0 + t \Lambda_1 ), \ t \in {\Z_{\ge 0}}.$ Therefore, $H_{\theta} ( \mathcal V_{-2} (D_{\ell}) )$ contains infinitely  many irreducible modules, which gives that $H_{\theta} ( \mathcal V_{-2} (D_{\ell}) ) =  V^ {\ell -4}(A_1) $. In this way we have proved claims (1) and (2).

 Let us now prove claim (3).

 Assume that $L (k\Lambda _0 + \mu)$  ($ \mu \in P_+$, $k=-2$) is an irreducible $\mathcal V_{k}(D_{\ell})$--module in the category $KL^ {k}$.
Then $H_{\theta} (L (k\Lambda _0 + \mu))$ is a non-trivial irreducible $V^{\ell-4}(A_1)$--module. The  representation theory of $V^ {\ell-4}(A_1)$  implies that:
$$H_{\theta} (L (k\Lambda _0 + \mu)) = L_{A_1} ( ( \ell  - 4 - j ) \Lambda_0 + j \Lambda_1) \qquad \mbox{for} \  j  \in {\Z}_{\ge 0}. $$
Since $D_\ell^{\natural} = A_1 \times  D_{\ell -2} $, we conclude that
$ \mu ^{\natural} =  j \omega_1  $ and therefore, by \eqref{mumus},
$$ \mu = j \omega_1 + s \omega_2  =  (s + j ) \epsilon_1  + s \epsilon _2 \qquad (s \in {\Z}_{\ge 0} ). $$
By using the  action of $L(0) =\omega_0$ on the lowest component of $H_{\theta} (L (k\Lambda _0 + \mu))$  we get
\bea \frac{ (\mu + 2 \rho, \mu)} {2 (k+h^{\vee})} - \mu (x)   = \frac{ j (j+2)}{4 (\ell - 2)} \qquad ( x = \theta ^{\vee} / 2).\nonumber \eea
Since $ 2 (k + h^{\vee}) =  2 ( -2 + 2 \ell -2) = 4 (\ell - 2)$  and $\mu (x)   =  (2 s + j ) / 2  $ we get
$$  (\mu + 2 \rho, \mu) -  (h^{\vee}-2) ( 2 s + j)  = j (j+2). $$
By direct calculation we get
$$   (\mu + 2 \rho, \mu)   = (s+j ) ^2 + s^2 + h^{\vee}  (s+j) + (h^{\vee} -2) s,$$
which gives an equation:
\bea  && (s+j ) ^2 + s^2 + h^{\vee}  (s+j) + (h^{\vee} -2) s  -   (h^{\vee}-2) ( 2 s + j)  = j (j+2). \nonumber \\
\iff  &&    (s+j ) ^2 + s^2 + h^{\vee}  (s+j)  -   (h^{\vee}-2) ( s + j)  = j (j+2).  \nonumber  \\
\iff && (s+j)  (s+j+ 2)  = j (j+2) \nonumber  \\
\iff && s = 0  \quad \mbox{or} \quad s = - 2 j -2 \nonumber . \eea
Since $ \mu \in P_+$ we conclude that $s= 0$. Therefore $\mu = j \omega_1$ for certain $j \in {\Z}_{\ge 0}$. The proof of claim (3) is now complete.
\end{proof}

 \subsection{The simple vertex algebra $V_{-2}( D_{\ell})$ }

 Next we use the fact that the simple affine $W$-algebra $W_{-2} (D_{\ell}, \theta)$ is isomorphic to the simple affine vertex algebra $V_{\ell-4} (A_1)$, for $\ell \geq 4$.
 \begin{prop} \label{class-d}
 The set
$\{ L( (-2-j ) \Lambda_0 + j \Lambda_1) \ \vert \ j \in {\Z_{\ge 0}}, j \le {\ell}-4 \}$
provides a complete list of irreducible  $ V_{-2} (D_{\ell})$--modules from the category $KL_{-2}$.
 \end{prop}
 \begin{proof}
 Assume that $N$ is an irreducible $V_{-2}(D_{\ell})$--module from the category $KL_{-2}$. Then $N$ is also irreducible as $\mathcal V_{-2} (D_{\ell})$--module, and therefore $N \cong
 L( (-2-j ) \Lambda_0 + j \Lambda_1) $ for certain $j \in {\Z}_{\ge 0}$. Since $H_{\theta} (N)$ must be an irreducible $H_{\theta} (V_{-2} (D_{\ell}) )= W_{-2} (D_{\ell}, \theta) = V_{\ell -4} (A_1)$--module, we get $ j \le \ell -4$, as desired.
 \end{proof}
 Now we want to describe the maximal ideal in $V^{-2}( D_{\ell})$. The next lemma states  that any non-trivial ideal in $\mathcal V_{-2} (D_{\ell})$ is automatically maximal.
 \begin{lem}Let $ \{ 0 \} \ne I   \subsetneqq  \mathcal V_{-2} (D_{\ell})$ be any non-trivial ideal in $\mathcal V_{-2} (D_{\ell})$. Then  we have
 \item[(1)]  $H_{\theta}  (I)$ is the  maximal ideal in $V^{\ell-4} (A_1)$.
 \item[(2)]   $I $ is a maximal ideal in $\mathcal V_{-2} (D_{\ell})$ and $I  = L ( - 2 (\ell-2) \Lambda_0 + 2 (\ell -3) \Lambda_1) $.
\end{lem}
\begin{proof}
  Assume that $I$ is a non-trivial ideal in   $\mathcal V_{-2} (D_{\ell})$. Then $I$ can be regarded as  a $\mathcal V_{-2} (D_{\ell})$--module in the category $KL^{-2}$ and therefore, by Proposition \ref{clas-2-intermediate}, (3),  it contains a non-trivial subquotient isomorphic to   $L( (-2-j ) \Lambda_0 + j \Lambda_1) $ for some $j \in {\Z}_{\ge 0}$. Since, by part (2) of the aforementioned Proposition,  $H_{\theta} (L( (-2-j ) \Lambda_0 + j \Lambda_1) ) \ne 0$ for every $j \in {\Z}_{\ge 0}$, we conclude that
 $H_{\theta}  (I )$ is a non-trivial ideal in $H_{\theta} (\mathcal V_{-2} (D_{\ell})) = V^{\ell -4} (A_1)$.  But since $ V^{\ell -4} (A_1)$, $\ell \geq 4$, contains a unique non-trivial ideal, which is automatically maximal,  we have that $H_{\theta}  (I )$ is a maximal ideal in $ V^{\ell -4} (A_1)$. So
$$H_{\theta}  (\mathcal V_{-2} (D_{\ell}  ) / I )\cong  V_{\ell -4}  (A_1). $$
Assume now that  $ \mathcal V_{-2} (D_{\ell}  ) / I  $ is not simple. Then it contains a non-trivial singular vector $v ' $ of weight $-(2+ j ) \Lambda_0 + j \Lambda_1$ for $j \in {\Z}_{> 0}$. By \cite{Araduke}, we have that $H_{\theta} (  V^{-2} (D_{\ell})  .v') $  is  a non-trivial ideal in $V_{\ell -4} (A_1)$ generated by  a singular vector of  $A_1$--weight $j \omega_1$. This is a contradiction.  So $I $ is  the maximal ideal.

Since the maximal ideal in $V^{\ell -4} (A_1)$ is generated by a singular vector of $A_1$--weight $ 2 (\ell-3) \omega_1$  and since the maximal ideal is simple,
we conclude that $I  = \mathcal V_{-2}(D_{\ell}). v_{sing}$ for a certain singular vector $v_{sing}$ of weight $\lambda =  - 2 (\ell-2) \Lambda_0 + 2 (\ell -3) \Lambda_1$.  It is also clear that this singular vector is unique, up to scalar factor. Therefore, $I = L ( - 2 (\ell-2)  \Lambda_0 + 2 (\ell -3) \Lambda_1)$.
\end{proof}
 Note that in the previous lemma we proved the existence of a  singular vector which generates the maximal ideal without presenting a formula for such a  singular vector.
 Since the vector in \eqref{w2} has the correct weight, we also have an explicit expression for this singular vector:
\bea
\Big(\sum _{i=2}^{ \ell} e_{\epsilon_1 - \epsilon_i}(-1)
e_{\epsilon_1 + \epsilon_i}(-1)\Big) ^{\ell -3} {\bf 1} \nonumber \eea
\begin{coro} \item[(1)] The maximal ideal in  $V^ {-2}(D_{\ell})$ is generated by the vectors $w_1$ and $w_2$  for $\ell \ge 5$ and by the vectors $w_1$, $w_2$, $w_3$ for $\ell = 4$.
\item[(2)] The homomorphism $\overline \Phi : \mathcal V_{-2} (D_{\ell} ) \rightarrow M_{2 \ell}$ is injective. In particular,  the vertex algebra
$\mathcal V_{-2} (D_{\ell}) \otimes V_{-\ell} (A_1)$ is conformally embedded into $V_{-1/2} (C_{2\ell})$.
\item [(3)] $\mbox{ch}  (\mathcal V_{-2} (D_{\ell} ) ) = \mbox{ch}  (V_{-2} (D_{\ell} ) ) +   \mbox{ch}  L ( - 2 (\ell-2) ) \Lambda_0 + 2 (\ell -3) \Lambda_1)$.
\end{coro}
 \begin{rem} D. Gaiotto in \cite{Ga} has started  a study of  the decomposition of $M_{2 \ell}$ as a $V^ { -2 } ( D_{ \ell} )  \otimes V_{-\ell} (A_1)  $--module in the case $\ell = 4$. By combining results from \cite[Section 8]{AKMPP-new} and results from this Section we get that
 $$\mbox{Com} (V_{-\ell} (A_1), M_{2 \ell}) \cong \mathcal V_{-2} (D_{\ell}).$$
 So the vertex algebra responsible for the  decomposition of $M_{2\ell}$ is exactly $\mathcal V_{-2} (D_{\ell})$. Therefore in the decomposition of  $M_{2\ell}$   only modules for
 $\mathcal V_{-2} (D_{\ell})$ can appear.  In our forthcoming papers we plan to apply  the representation theory of $\mathcal V_{-2}(D_{\ell})$  to the problem of finding branching rules.
 \end{rem}
\begin{coro} \label{semi-simple-d}
For $\ell \ge 3$   the category $KL_{-2}$ is semi-simple.
\end{coro}
 \begin{proof}
  The assertion in the case $ \ell \ge 4$ follows from  Theorem  \ref{rational} since then $W_{-2} (D_{\ell}, \theta) = V_{\ell -4} (sl(2))$ is a rational vertex algebra.

   In the case $\ell = 3$, we have that a highest weight $V_{-2}(D_3)$--module $M$  is isomorphic to $ \widetilde   L( (-2-j ) \Lambda_0 + j \Lambda_1) $ where  $j \in {\Z}_{\ge 0}$. The irreducibility of $M$ follows easily from  the fact  that $H_{\theta} (M )$ is isomorphic to an irreducible $V_{-1} (sl(2))$--module    $L _{A_1} (-1-j ) \Lambda_0 + j \Lambda_1) $.  Now claim follows from Theorem  \ref{general-complete-reducibility} and Lemma \ref{kriterij}.
 \end{proof}

\section{ The vertex algebra $V^{-2} (B_{\ell} ) $ and its quotients  }

In this section let $\ell \ge 2$.
Note that $k=-2$ is a collapsing level for $B_{\ell}$ \cite{AKMPP-JA}, and that the simple affine $W$-algebra $W_{-2} (B_{\ell}, \theta)$ is isomorphic to $V_{\ell -\frac{7}{2}} (A_1)$.
This implies that $H_{\theta} (V_{-2} (B_{\ell})) = V_{\ell -\frac{7}{2}} (A_1). $
But as in the case of the affine Lie algebra of type $D$,  we can construct an intermediate vertex algebra $\mathcal V $ so that   $H_{\theta} ( \mathcal V ) = V^{\ell -7/2}(A_1)$.
\begin{rem} \label{remark-singv-notation} The formula for a singular vector of conformal weight two in $V^{-2} (B_{\ell})$ was given in \cite[Theorem 4.2]{AM}
for $\ell \ge 3$, and in \cite[Remark 4.3]{AM} for $\ell =2$. Note that, for $\ell \ge 4$, the vector $\sigma(w_2)$ from \cite{AM} is equal to the vector
$w_1$ from relation (\ref{sing-D4-1}), i.e. it is contained in the subalgebra $V^{-2}(D_4)$. For $\ell =3$, we have
\bea w_1 = (e_{\epsilon_1 + \epsilon_2}(-1) e_{\epsilon_3}(-1)
- e_{\epsilon_1 + \epsilon_3}(-1) e_{\epsilon_2}(-1)
 + e_{\epsilon_1}(-1)
e_{\epsilon_2 + \epsilon_3}(-1) ) {\bf 1}. \nonumber
\eea
For $\ell =2$, the singular vector of conformal weight two in $V^{-2} (B_2)$ is equal to
\bea w_1 = (e_{\epsilon_1 + \epsilon_2}(-1) e_{-\epsilon_2}(-1)
+ \frac{1}{2}h_{\epsilon_2}(-1) e_{\epsilon_1}(-1) - e_{\epsilon_1 - \epsilon_2}(-1)
e_{\epsilon_2} (-1) ) {\bf 1}. \nonumber
\eea
\end{rem}
Consider the singular vector in $V^{-2} (B_{\ell} ) $ denoted by $\sigma(w_2)$  in \cite[Theorem 4.2]{AM} and \cite[Section 7]{AM2}.
Let us denote that singular vector by $w_1$ in this paper (see Remark \ref{remark-singv-notation} for explanation).

Then we have the quotient vertex algebra
\begin{equation}\label{v-2b}\mathcal V_{-2} (B_{\ell}) = V^{-2} (B_{\ell})\big/ \langle  w_1 \rangle. \end{equation}
As in the case of the vertex algebra $\mathcal V_{-2} (D_{\ell})$, we have the non-trivial homomorphism $\mathcal V_{-2} (B_{\ell}) \rightarrow  M_{2\ell +1}$.

The proof of the following result is completely analogous to the proof of Proposition \ref{clas-2-intermediate} and it is therefore omitted.
\begin{prop} \label{clas-2-intermediate-B}
We have
\item[(1)] There is a non-trivial homomorphism $\overline \Phi : \mathcal V_{-2}  (B_{\ell})  \rightarrow M_{2\ell +1}$.
\item[(2)] $H_{\theta} ( \mathcal V_{-2} (B_{\ell}) ) = V^{\ell -7/2}(A_1)$.
\item[(3)] $ H_{\theta}( L( (-2-t) \Lambda_0 + t \Lambda_1) ) \cong L_{A_1}( (\ell-7/2 -t) \Lambda_0 + t \Lambda_1 ), \ t \in {\Z_{\ge 0}}.$
\item[(4)] The set
\bea \label{clas-b} \{ L( (-2-t) \Lambda_0 + t \Lambda_1) \ \vert \ t \in {\Z_{\ge 0}} \}  \eea
provides a complete list of irreducible  $\mathcal V_{-2} (B_{\ell})$--modules from the category $KL^{-2}$.
\end{prop}
We have the following result on classification of irreducible modules.
 \begin{prop} \label{class-b}
 Assume that $\ell \ge 3$. Then the  set
$\{ L( (-2-j ) \Lambda_0 + j \Lambda_1) \ \vert \ j \in {\Z_{\ge 0}}, j \le 2 ( {\ell}-3 ) + 1 \}$
provides a complete list of irreducible  $ V_{-2} (B_{\ell})$--modules from the category $KL_{-2}$.
 \end{prop}
 \begin{proof}
The proof is analogous  to the proof of Proposition \ref{class-d}: it uses  the  exactness of the  functor $H_{\theta}$ and the representation theory of affine vertex algebras. In particular, we use the result from \cite{AdM} which gives that  the set
$$\{ L( -(\ell - 7/2) -j ) \Lambda_0 + j \Lambda_1) \ \vert \ j \in {\Z_{\ge 0}}, j \le 2 ( {\ell}-3 ) + 1 \}$$ provides a complete list of irreducible  $ V_{\ell - 7/2 } (A_1)$--modules from the category $KL_{\ell -7/2}$.
 \end{proof}
An  important consequence is the simplicity of the vertex algebra  $ \mathcal V_{-2} (B_{2 })$.
\begin{coro}\label{coro-simpl} The vertex algebra $ \mathcal V_{-2}  (B_{\ell})  $ is simple if and only if $\ell =2$.  In particular, the set (\ref{clas-b}) provides a complete list of irreducible modules for $V_{-2} (B_2)$ in $KL_{-2}$.
\end{coro}
\begin{proof}
Since  by Proposition \ref{clas-2-intermediate-B}, $\mathcal V_{-2} (B_{\ell})$ has infinitely many irreducible modules in the category $KL^{-2}$, and, by Proposition \ref{class-b},      $ V^{-2} (B_{\ell})$  has finitely many irreducible modules in the category $KL^{-2}$ (if $\ell \ge 3$), we conclude that  $ \mathcal V_{-2}  (B_{\ell})  $ cannot be simple for $\ell \ge 3$.

Let us consider the case $\ell = 2$.  Assume that $\mathcal V_{-2} (B_2)$ is not simple. Then it must contain an ideal  $I$ generated by a singular vector of weight $\lambda= (-2-j ) \Lambda_0 + j \Lambda_1$ for certain $j >0$. By applying the functor $H_{\theta}$, we get a non-trivial ideal in $V^{-3/2} (A_1)$, against the  simplicity of $V^{-3/2} (A_1)$.\end{proof}
Next we notice that $ V^{\ell -7/2}(A_1)$ has a unique non-trivial ideal $J$ which is generated by a singular vector of $A_1$--weight  $4(\ell-2) \omega_1$. The ideal $J$ is maximal and simple (cf. \cite{AKMPP-JJM}).  By combining this with  properties of the functor $H_{\theta}$ from \cite{Araduke},  one proves the existence of a unique maximal ideal $I$ (which is also simple) in $\mathcal V_{-2} (B_{\ell})$   such that $ I  \cong  L( -2 (2\ell -3) \Lambda_0 + 4(\ell -2)  \Lambda_1) )$.
\begin{rem} The explicit expression for a singular vector which generates $I$ is more complicated that in the case $D$, and it won't be presented here. \end{rem}
In \cite{AKMPP-new} we constructed a homomorphism   $\mathcal V_{-2} (B_{\ell}) \otimes V_{-\ell-1/2} (A_1) \rightarrow M_{2 \ell +1}$. The results of this section enable us to find the image of this homomorphism.
\begin{coro} We have:
\item[(1)] The vertex algebra
$\mathcal V_{-2} (B_{\ell}) \otimes V_{-\ell- 1/2} (A_1)$ is conformally embedded into $V_{-1/2} (C_{2\ell +1})$.
\item[(2)]  The vertex algebra $\mathcal V_{-2} (B_{\ell})$ for $\ell \ge 3$ contains a unique ideal $ I \cong L( -2 (2\ell -3) \Lambda_0 + 4 (\ell -2)  \Lambda_1) $ and
 $$\mbox{ch}  (\mathcal V_{-2} (B_{\ell} ) ) = \mbox{ch}  (V_{-2} (B_{\ell} ) ) +   \mbox{ch}  (L (-2 (2\ell -3) \Lambda_0 + 4(\ell -2)  \Lambda_1)).$$
\end{coro}
Finally, we apply Theorem
 \ref{general-complete-reducibility} and prove that $KL_{-2}$ is a semi-simple category.
 \begin{coro} \label{semi-simple-b}
If  $\ell \ge 2$, then every $V_{-2}(B_{\ell})$--module in $KL_{-2}$ is  completely  reducible.
 \end{coro}
\begin{proof}
 It suffices to prove that  every highest weight  $V_{-2}(B_{\ell})$--module  in $KL_{-2}$ is irreducible.   Assume that $\ell \ge 3$.  If $M \cong  \widetilde L(\lambda)$  is a highest weight module in $KL_{-2}$ then the highest weight is $\lambda = -(2+ j) \Lambda _0  + j \Lambda_1$ where $0 \le j \le 2 (\ell -3) + 1$.
 Since $H_{\theta} (L(\lambda) ) $ is a non-zero  highest weight $V_{-\ell + 7/2} (sl(2))$--module, then the complete reducibility result from \cite{AdM} implies that $H_{\theta} (L(\lambda) ) $ is irreducible. The assertion now follows from  Lemma  \ref{kriterij}.
 The proof in the case $\ell =2$ is similar, and it uses the classification of irreducible $V_{-2} (B_2)$--modules from Corollary \ref{coro-simpl} and the fact that every  highest weight $V_{-3/2} (sl(2)) = H_{\theta} (V_{-2} (B_2))$--module in $KL_{-3/2}$ is irreducible.
\end{proof}

\section{On the representation theory of $V_{2-\ell} (D_{\ell} )$ } \label{rep-th-D}

\subsection{ The vertex algebra $\overline V_{2-\ell} (D_{\ell} )$} \label{subsect-V-bar}  Let ${\mathfrak g}$ be a simple Lie algebra of type $D_{ \ell}$.
Recall that $2-\ell = -h^{\vee} / 2 + 1$ is a collapsing level \cite{AKMPP-JA}. We have the singular vector
\bea v_n = \Big(\sum _{i=2}^{ \ell} e_{\epsilon_1 - \epsilon_i}(-1)
e_{\epsilon_1 + \epsilon_i}(-1)\Big) ^n {\bf 1} \label{singvect-D-old} \eea
in $V^ { n - \ell +1 } ( D_{ \ell} )  $, for any $n \in \Z _{>0}$. As in \cite{P}, we consider the vertex algebra
 \begin{equation}\label{vbar} \overline  V_{2-\ell} (D_{\ell} ) =  V^{2-\ell} (D_{\ell} )\big/\langle v_1 \rangle ,\end{equation}
 where  ${\langle v_1 \rangle}$ denotes the ideal in  $V^{2-\ell} (D_{\ell} )$ generated by the singular vector $v_1$. We recall the following result on the classification of irreducible  $  \overline  V_{2-\ell} (D_{\ell} )$--modules in the category $KL^{2-\ell}$.
 \begin{prop}\cite{P} \label{prop-rep-th-D}
 \item[(1)] The set $$\{ V(t \omega_{\ell}), V(t \omega_{\ell-1}) \ \vert \ t \in {\Z_{\ge 0}} \}$$ provides a complete list of irreducible finite-dimensional modules for the  Zhu algebra
 $A(\overline  V_{2-\ell} (D_{\ell} ) ). $
 \item[(2)]
 The set $$\{ L( (2-t-\ell) \Lambda_0 + t \Lambda_{\ell}), L ( (2-t-\ell) \Lambda_0+ t \Lambda_{\ell-1}) \ \vert \ t \in {\Z_{\ge 0}} \}$$ provides a complete list of irreducible  $\overline V_{2-\ell} (D_{\ell})$--modules from the category $KL^{2-\ell}$.
 \end{prop}
In the odd rank case $D_{2 \ell -1}$, the modules from Proposition \ref{prop-rep-th-D} (2) provide
a complete list of irreducible $V_{3-2\ell} (D_{2\ell -1} )$--modules from the category $KL_{3-2 \ell}$ (cf. \cite{AP}).
The  paper \cite{AP} also contains a fusion rules result in the category $KL_{3-2 \ell}$.
Detailed fusion rules analysis will be presented   elsewhere.

On the other hand, Theorem \ref{thm-classification-unique} implies that in the even rank case $D_{2 \ell}$,
$V_{2-2\ell} (D_{2\ell} )$ is the unique irreducible $V_{2-2\ell} (D_{2\ell} )$--module from the category $KL_{2-2\ell}$.
In the next section we will give an explanation of this difference using singular vectors existing in the
even rank case $D_{2 \ell}$.

\subsection{Singular vectors in $V^{n-2 \ell+1}(D_{2 \ell})$} \label{sect-sigv-even}

In this section, we construct more singular vectors in $V^{n-2 \ell+1}(D_{2 \ell})$. In the case
$n=1$, we show that the maximal submodule of $V^{2-2 \ell}(D_{2 \ell})$ is generated
by three singular vectors. We present explicit formulas for these singular vectors.

Let ${\mathfrak g}$ be a simple Lie algebra of type $D_{2 \ell}$. Denote by $S_{2 \ell}$ the group of
permutations of $2 \ell$ elements. Let
$$\Pi _{\ell}= \Big\{ p \in S_{2 \ell}  \ \vert \ p^2=1, \ p(i) \neq i, \forall i \in \{1, \ldots, 2 \ell \}  \Big\}$$
be the set of fixed-points free involutions, which is well known to have  $(2 \ell -1)!!= 1 \cdot 3 \cdot \ldots \cdot (2\ell-1)$  elements.
For $i \neq j$, denote by $(i\,j) \in S_{2 \ell}$ the transposition of $i$ and $j$. Then, any $p \in \Pi_\ell$ admits a unique
decomposition of the form:
$$ p=(i_1\,j_1)\cdots(i_\ell\,j_\ell), $$
such that $i_h<j_h$ for $1\leq h\leq \ell$, and $i_1 < \ldots < i_{\ell}$. Define a permutation $\bar{p} \in S_{2 \ell}$ by:
$$ \bar{p}(2h-1)= i_h, \ \bar{p}(2h)= j_h, \,1\leq h\leq \ell. $$
Thus, we have a well defined map $p \mapsto \bar{p}$ from $\Pi_\ell$ to $S_{2 \ell}$.
Define the function  $s: \Pi_\ell \to \{ \pm 1 \}$ as follows:
$$ s(p) = \mbox{sign} (\bar{p}),$$
where $\mbox{sign} ( q )$ denotes the sign of the permutation $q \in S_{2 \ell}$.

We have:
\begin{thm} \label{sing-vectD-novi}
The vector
\begin{equation}\label{wn}w_n = \Big( \sum _{p \in \Pi _{\ell}} s(p) \prod_{ { i \in \{1, \ldots , 2 \ell \} } \atop {i < p(i)} } e_{\epsilon_i + \epsilon_{p(i)}} (-1) \Big) ^n {\bf 1} \end{equation}
is a singular vector in $V^{n-2 \ell+1}(D_{2 \ell})$, for any $n \in \Z _{>0}$.
\end{thm}
\begin{proof} Direct verification of relations $e_{\epsilon_k -
\epsilon_{k+1}}(0)w_n=0$, for $k=1,\ldots ,2\ell-1$,
$e_{\epsilon_{2\ell-1} + \epsilon_{2\ell}}(0)w_n=0$ and $e_{-(\epsilon_1
+ \epsilon_2)}(1)w_n=0$.
\end{proof}
\begin{rem} The vector $w_n$ has conformal weight $n \ell$ and its  $\g$--highest weight equals
$2n \omega _{2 \ell}= n (\epsilon _1 + \ldots + \epsilon _{2 \ell})$. In particular, for $n=1$,
the vector $w_1$ has conformal weight $\ell$ and  highest weight
$2 \omega _{2 \ell}= \epsilon _1 + \ldots + \epsilon _{2 \ell}$.
\end{rem}
\begin{ex} \label{ex-sing-v} Set $n=1$ for simplicity. For $\ell =2$ we recover the singular vector
$$ w_1= (e_{\epsilon_1 + \epsilon_2}(-1)
e_{\epsilon_3 + \epsilon_4}(-1)-  e_{\epsilon_1 + \epsilon_3}(-1)
e_{\epsilon_2 + \epsilon_4}(-1) + e_{\epsilon_1 + \epsilon_4}(-1)
e_{\epsilon_2 + \epsilon_3}(-1) ) {\bf 1}$$
in $V^{-2}(D_{4})$ of conformal weight $2$ from \cite{P}. For $\ell =3$, the formula for the singular vector in
$V^{-4}(D_{6})$ of conformal weight $3$ is more complicated. It is a sum of $5!!=15$ monomials:
\bea
&& w_1= ( e_{\epsilon_1 + \epsilon_2}(-1) e_{\epsilon_3 + \epsilon_4}(-1) e_{\epsilon_5 + \epsilon_6}(-1) -
e_{\epsilon_1 + \epsilon_2}(-1) e_{\epsilon_3 + \epsilon_5}(-1) e_{\epsilon_4 + \epsilon_6}(-1) \nonumber \\
&& + e_{\epsilon_1 + \epsilon_2}(-1) e_{\epsilon_3 + \epsilon_6}(-1) e_{\epsilon_4 + \epsilon_5}(-1) -
e_{\epsilon_1 + \epsilon_3}(-1) e_{\epsilon_2 + \epsilon_4}(-1) e_{\epsilon_5 + \epsilon_6}(-1) \nonumber \\
&&  + e_{\epsilon_1 + \epsilon_3}(-1) e_{\epsilon_2 + \epsilon_5}(-1) e_{\epsilon_4 + \epsilon_6}(-1)
- e_{\epsilon_1 + \epsilon_3}(-1) e_{\epsilon_2 + \epsilon_6}(-1) e_{\epsilon_4 + \epsilon_5}(-1) \nonumber \\
&&  + e_{\epsilon_1 + \epsilon_4}(-1) e_{\epsilon_2 + \epsilon_3}(-1) e_{\epsilon_5 + \epsilon_6}(-1) -
e_{\epsilon_1 + \epsilon_4}(-1) e_{\epsilon_2 + \epsilon_5}(-1) e_{\epsilon_3 + \epsilon_6}(-1)    \nonumber \\
&&  + e_{\epsilon_1 + \epsilon_4}(-1) e_{\epsilon_2 + \epsilon_6}(-1) e_{\epsilon_3 + \epsilon_5}(-1) -
e_{\epsilon_1 + \epsilon_5}(-1) e_{\epsilon_2 + \epsilon_3}(-1) e_{\epsilon_4 + \epsilon_6}(-1) \nonumber \\
&&  + e_{\epsilon_1 + \epsilon_5}(-1) e_{\epsilon_2 + \epsilon_4}(-1) e_{\epsilon_3 + \epsilon_6}(-1)
- e_{\epsilon_1 + \epsilon_5}(-1) e_{\epsilon_2 + \epsilon_6}(-1) e_{\epsilon_3 + \epsilon_4}(-1) \nonumber \\
&&   + e_{\epsilon_1 + \epsilon_6}(-1) e_{\epsilon_2 + \epsilon_3}(-1) e_{\epsilon_4 + \epsilon_5}(-1)
- e_{\epsilon_1 + \epsilon_6}(-1) e_{\epsilon_2 + \epsilon_4}(-1) e_{\epsilon_3 + \epsilon_5}(-1)  \nonumber \\
&&  + e_{\epsilon_1 + \epsilon_6}(-1) e_{\epsilon_2 + \epsilon_5}(-1) e_{\epsilon_3 + \epsilon_4}(-1) )  {\bf 1}. \nonumber
\eea
\end{ex}
Denote by $\vartheta$ the automorphism of $V^{n-2 \ell+1}(D_{2 \ell})$ induced by
the automorphism of the Dynkin diagram of $D_{2 \ell}$ of order two such
that
\bea
&& \vartheta (\epsilon_k -
\epsilon_{k+1}) = \epsilon_k -
\epsilon_{k+1}, \  k=1,\ldots ,2\ell-2,  \label{teta1} \\
&& \vartheta
(\epsilon_{2\ell-1} - \epsilon_{2\ell}) = \epsilon_{2\ell-1} + \epsilon_{2\ell}, \ \vartheta
(\epsilon_{2\ell-1} + \epsilon_{2\ell}) = \epsilon_{2\ell-1} - \epsilon_{2\ell}. \label{teta2} \eea

Theorem \ref{sing-vectD-novi} now implies that $\vartheta (w_n)$ is  a
singular vector  in $V^{n-2 \ell+1}(D_{2 \ell})$, for any $n \in \Z _{>0}$, also.
The vector $\vartheta (w_n)$ has conformal weight $n \ell$ and its  highest weight for ${\mathfrak g}$ is
$2n \omega _{2 \ell -1}= n (\epsilon _1 + \ldots + \epsilon _{2 \ell -1}- \epsilon _{2 \ell})$.

We consider the associated quotient vertex algebra
\begin{equation}\label{Vtilde}\widetilde{V}_{n-2\ell+1}(D_{2\ell}):= V^{n-2 \ell+1}(D_{2 \ell})
\big/ \langle v_n, w_n, \vartheta (w_n) \rangle,\end{equation}
where $v_n$ is given by relation (\ref{singvect-D-old}) (for $D_{2 \ell}$):
\bea v_n = \Big(\sum _{i=2}^{2 \ell} e_{\epsilon_1 - \epsilon_i}(-1)
e_{\epsilon_1 + \epsilon_i}(-1)\Big) ^n {\bf 1}. \nonumber \eea

In particular, for $n=1$ we have the vertex algebra
$$\widetilde{V}_{2-2\ell}(D_{2\ell})=V^{2-2\ell}(D_{2\ell})\big/
\langle v_1, w_1, \vartheta (w_1)\rangle.$$
Clearly, $\widetilde{V}_{2-2\ell}(D_{2\ell})$ is a
quotient of vertex algebra $\overline  V_{2-2\ell} (D_{2\ell} )$ from Subsection~\ref{subsect-V-bar}.
The associated Zhu algebra is
$$A ( \widetilde{V}_{2-2\ell}(D_{2\ell}) )= U({\mathfrak g})\big/
\langle \bar{v}, \bar{w}, \vartheta (\bar{w})\rangle,$$
where
$$ \bar{v} = \sum _{i=2}^{2 \ell} e_{\epsilon_1 - \epsilon_i}
e_{\epsilon_1 + \epsilon_i},\qquad \bar{w} =  \sum _{p \in \Pi _{\ell}} s(p) \prod_{ { i \in \{1, \ldots , 2 \ell \} }
\atop {i < p(i)} } e_{\epsilon_i + \epsilon_{p(i)}}. $$
\begin{lem} \label{lem-klasif}  \label{pomoc-11} We have:
\begin{itemize}
\item[(1)] $\bar{w} V(t \omega _{2 \ell}) \neq 0$, for $t \in \Z _{>0}$.
\item[(2)] $\vartheta (\bar{w}) V(t \omega _{2 \ell -1}) \neq 0$, for $t \in \Z _{>0}$.
\end{itemize}
\end{lem}
\begin{proof} (1) Let $t=1$. Denote by $v_{\omega _{2 \ell}}$ the highest weight vector of $V(\omega _{2 \ell})$,
and by $v_{- \omega _{2 \ell}}$ the lowest weight vector of $V(\omega _{2 \ell})$. One can easily check, using the
spinor realization of $V(\omega _{2 \ell})$, that there exists a constant $C \neq 0$ such that
$$ \bar{w}( v_{- \omega _{2 \ell}}) = C v_{\omega _{2 \ell}}. $$
For general $t \in \Z _{>0}$, the claim follows using the embedding of $V(t \omega _{2 \ell})$ into
$V( \omega _{2 \ell}) ^{\otimes t}$. Claim (2) follows similarly.
\end{proof}
\begin{thm} \label{class-d-direct} We have:
\item[(i)]
The trivial module $\C$ is the unique finite-dimensional irreducible module for
$A ( \widetilde{V}_{2-2\ell}(D_{2\ell}) )$.
\item[(ii)]
$V_{2-2\ell}(D_{2\ell})$ is the unique irreducible $\g$--locally finite module for
$\widetilde{V}_{2-2\ell}(D_{2\ell})$.
\item[(iii)] The vertex operator algebra $\widetilde{V}_{2-2\ell}(D_{2\ell})$ is simple,
i.e.
$$V_{2-2\ell}(D_{2\ell})= V^{2-2\ell}(D_{2\ell})\big/
\langle  v_1,  w_1, \vartheta (w_1)\rangle.$$
\end{thm}
\begin{proof} (i) Proposition \ref{prop-rep-th-D} implies that the set
$$ \{ V(t \omega _{2 \ell}), V(t \omega _{2 \ell -1}) \ \vert \ t \in \Z _{\geq 0}  \} $$
provides a complete list of finite-dimensional irreducible modules for the algebra
$U({\mathfrak g})\big/\langle\bar{v}\rangle = A(\overline V_{2-2\ell} (D_{2\ell} ) ). $
Lemma \ref{lem-klasif} shows that $V(t \omega _{2 \ell})$ and $V(t \omega _{2 \ell -1})$ are not modules for
$A ( \widetilde{V}_{2-2\ell}(D_{2\ell}) )$, for $t \in \Z _{>0}$. Claim (i) follows. Claims (ii) and (iii) follow from (i) by applying Proposition \ref{zhu-int} and Corollary \ref{simplicity-1}.
\end{proof}
 \begin{rem}
 A general  character formula for certain simple  affine vertex algebras at negative integer levels has been recently  presented  by V. G. Kac and M. Wakimoto in \cite{KW3},
(more precisely, $\g=A_n, C_n$ for $k=-1$ and  ${\mathfrak g}= D_4, E_6, E_7, E_8$ for $k=-2,-3,-4,6$). Note that conditions (i)-(iii) of \cite[Theorem 3.1]{KW3} hold for vertex algebras $V_{-b}( D_{n})$, $n>4$, $b = 1, \dots , n -2$, too. We conjecture that condition (iv) of this theorem holds as well;  therefore  formula (3.1) in  \cite{KW3}  gives the  character formula.
 \end{rem}

\section{Conformal embedding of $\widetilde V(-4, D_6\times A_1) $ into $V_{-4}(E_7)$ }

In this section, we apply the results on representation theory of $V_{-4}(D_6)$ from previous sections
to the conformal embedding of $\widetilde V(-4, D_6\times A_1)$ into $V_{-4}(E_7)$.
This gives us an interesting example of a maximal semisimple equal rank subalgebra
such that the associated conformally embedded subalgebra
is not simple.

We use the construction of the root system of type $E_7$ from
\cite{Bou}, \cite{H}, and the notation for root vectors similar to the
notation for root vectors for $E_6$ from \cite{AP-ART}.

For a subset $S= \{i_1, \ldots ,i_k \} \subseteq \{1,2,3,4,5,6 \}$,
$i_1 < \ldots < i_k$, with odd number of
elements (so that  $k=1,3$ or $5$), denote by $e_{(i_1 \ldots i_k)}$ a
suitably chosen root vector associated to the positive root
$$ \frac{1}{2}\left(\epsilon_8 - \epsilon_7 + \sum
_{i=1}^{6} (-1)^{p(i)} \epsilon_i \right), $$
such that $p(i)=0$ for $i \in S$ and $p(i) =1$ for $i \notin S$. We
will use the symbol $f_{(i_1 \ldots i_k)}$ for the root vector
associated to corresponding negative root.

Note now that the subalgebra of $E_7$
generated by positive root vectors
\begin{equation}
e_{\epsilon_6 + \epsilon_5},
  e_{\alpha _1}= e_{(1)},
 e_{\alpha _3}= e_{\epsilon _2 - \epsilon _1},
 e_{\alpha _4}= e_{\epsilon _3 - \epsilon _2},
 e_{\alpha _2}= e_{\epsilon _1 + \epsilon _2},
 e_{\alpha _5}= e_{\epsilon _4 - \epsilon _3}
\end{equation}
and the associated negative root vectors is a simple Lie algebra of type
$D_6$. There are $30$ root vectors
associated to positive roots for $D_6$:
\bea  && e_{\epsilon_6 + \epsilon_5}, \ e_{\epsilon_8 - \epsilon_7},\nonumber \\
&&  e_{(i)}, \ i \in \{1,2,3,4 \}, \nonumber \\
&& e_{(ijk)}, \ i,j,k \in \{1,2,3,4 \}, \ i<j<k, \nonumber \\
&& e_{(i56)}, \ i \in \{1,2,3,4 \}, \nonumber \\
&& e_{(ijk56)}, \ i,j,k \in \{1,2,3,4 \}, \ i<j<k, \nonumber \\
&& e_{ \pm \epsilon_i + \epsilon_j}, \ i,j \in \{1,2,3,4 \}, i<j.
 \eea
Furthermore, the subalgebra of $E_7$
generated by  $e_{\epsilon _6 - \epsilon _5}$
and the  associated negative root vector is a simple Lie algebra of type
$A_1$. Thus, $D_6 \oplus A_1$  is a semisimple subalgebra of $E_7$.

It follows from \cite{AKMPP}, \cite{AP-ART} that the affine vertex  algebra
$\widetilde V(-4, D_6\times A_1) $
is conformally embedded in $V_{-4}(E_7)$.
 Remark that $\widetilde V(-4,A_1)=V_{-4}(A_1)$ (since $V^{-4} (A_1) = V_{-4} (A_1)$).   This implies    that $\widetilde V(-4,D_6\times A_1) \cong
\widetilde V(-4,D_6)\otimes V_{-4}(A_1)$.

It was shown in \cite{AM} that
\begin{eqnarray} \label{sing-vect-E7}
&&v_{E_{7}} = ( e_{\epsilon_8 - \epsilon_7}(-1) e_{\epsilon_6 + \epsilon_5}(-1) +  e_{(156)}(-1)e_{(23456)}(-1) + \nonumber \\
&&+ e_{(256)}(-1)e_{(13456)}(-1) + e_{(356)}(-1)e_{(12456)}(-1)+ \nonumber \\ &&+e_{(456)}(-1)e_{(12356)}(-1) ) {\bf 1} \end{eqnarray}
is a singular vector in $V^{-4}(E_7)$. Moreover,
$$V_{-4}(E_7) \cong V^{-4}(E_7)\big/\langle v_{E_{7}}\rangle.$$

Vectors $(e_{(12346)}(-1))^s {\bf 1}$, for $s \in \Z _{>0}$ are (non-trivial) singular vectors for
the affinization of $D_6 \oplus A_1$ in $V_{-4}(E_7)$ of highest
weights $-(s+4)\Lambda_0+ s \Lambda _6$ for $D_6 ^{(1)}$ and $-(s+4)\Lambda_0+ s
\Lambda _1$ for $A_1 ^{(1)}$.
Thus there exist highest weight modules $\widetilde{L}_{D_6} (-(s+4)\Lambda_0+ s \Lambda _6)$
and  $ \widetilde{L}_{A_1} (-(s+4)\Lambda_0+ s \Lambda _1)$, for $D_6 ^{(1)}$ and $A_1 ^{(1)}$, respectively
such that
$(\widetilde V(-4,D_{6}) \otimes V_{-4}(A_1) ). (e_{(12346)}(-1))^s {\bf 1}$ is isomorphic to $\widetilde{L}_{D_6} (-(s+4)\Lambda_0+ s \Lambda _6) \otimes
\widetilde{L}_{A_1} (-(s+4)\Lambda_0+ s \Lambda _1).$
This implies that
$$ \ L_{D_6} (-(s+4)\Lambda_0+ s \Lambda _6) \otimes
L_{A_1} (-(s+4)\Lambda_0+ s \Lambda _1)  $$
are irreducible $\widetilde V(-4, D_6\times A_1) $--modules,
for $s \in \Z _{>0}$. \par
In particular,
$ L_{D_6} (-(s+4)\Lambda_0+ s \Lambda _6)$ are irreducible ($D_6$--locally finite)
$\widetilde V(-4, D_6)$--modules, for $s \in \Z _{>0}$.
In the next proposition, we use the notation from  \eqref{vbar}, \eqref{wn}, \eqref{teta1}, \eqref{teta2}.
\begin{prop} We have:
\item [(1)]  Assume that $  \widetilde{L}_{D_6} (-6 \Lambda_0+ 2 \Lambda _6)  $ and  $\widetilde{L}_{D_6} (-6 \Lambda_0+ 2 \Lambda _5)  $ are highest weight  $\overline  V_{-4} (D_6)$--modules  from the category $KL^{-4}$, not necessarily irreducible. Then
$$\widetilde{L}_{D_6} (-6 \Lambda_0+ 2 \Lambda _6)  \boxtimes \widetilde{L}_{D_6} (-6 \Lambda_0+ 2 \Lambda _5)   = 0, $$
where $\boxtimes$ is the  tensor  functor for $KL^{-4}$--modules. In other words, we cannot have  a non-zero    $\overline  V_{-4} (D_6)$--module $M$ from $KL^{-4}$ and  a non-zero  intertwining operator of type
\bea   { M \choose   \widetilde{L}_{D_6} (-6 \Lambda_0+ 2 \Lambda _6)  \quad \widetilde{L}_{D_6} (-6 \Lambda_0+ 2 \Lambda _5)  } . \label{int} \eea
\item[(2)]    Relations $w_1 \ne 0$ and $\vartheta (w_1) = 0$ hold in $V_{-4}(E_7)$. In particular, $\widetilde{V}(-4,D_{6})$ is not simple.
\end{prop}
\begin{proof}
 For the proof of assertion (1) we first notice that the following  decomposition of $D_6$--modules holds:
\bea V_{D_6} (2 \omega_6) \otimes V_{D_6}( 2 \omega_5) &=& V_{D_6} (2 \omega_5 + 2 \omega_6) \oplus  V_{D_6} ( \omega_3 + \omega_5 +  \omega_6) \oplus V_{D_6} (2 \omega_3 )\nonumber \\
&& \oplus V_{D_6} ( \omega_1 + \omega_5 +  \omega_6) \oplus V_{D_6} (\omega_1 + \omega_3)  \oplus V_{D_6} (2 \omega_1). \label {decomo56} \eea
Assume that  $M$ is a non-zero $\overline  V_{-4} (D_6)$--module  in the category  $KL^{-4}$ such that there is a non-trivial intertwining operator of type (\ref{int}). Then the  Frenkel-Zhu formula for fusion rules implies that $M$ must contain a non-trivial subquotient whose lowest graded component appears in the decomposition  of $V_{D_6} (2 \omega_6) \otimes V_{D_6}( 2 \omega_5) $. But by Proposition  \ref{prop-rep-th-D},  the $D_6$--modules appearing in  (\ref{decomo56}) cannot be lowest components of  any $\overline  V_{-4} (D_6)$--module. This proves  assertion (1).

Assertion (1) implies that if $w_1 \ne 0$ and $\vartheta (w_1) \ne 0$ in  $V_{-4}(E_7)$, then $$Y(w_1, z) \vartheta (w_1)  = 0, $$ a contradiction since $V_{-4}(E_7)$ is a simple vertex algebra.   The same fusion rules  argument shows that if $\vartheta (w_1) \ne 0$ in  $V_{-4}(E_7)$, then
$$Y(\vartheta (w_1), z) e_{(12346)}(-1)^2 {\bf 1} = 0 , $$
which again contradicts the simplicity of $V_{-4}(E_7)$. So,  $\vartheta (w_1)  = 0$.

 But if $w_1 = 0$, then, by Theorem \ref{class-d-direct} (iii), we have that $\widetilde{V}(-4,D_{6}) = V_{-4} (D_6)$.   Theorem \ref{thm-classification-unique} implies
that $\widetilde{V}(-4,D_{6})$ is not simple, since the simple vertex
operator algebra $V_{-4}(D_6)$ has only one irreducible $D_6$--locally finite module, a contradiction. So $w_1 \ne 0$ and  claim (2) follows.
\end{proof}
Set
\begin{equation}\label{vmenoquattro}
\mathcal {V}_{-4}(D_{6}) =  \frac{V^{-4}(D_6)} { < v_1, \vartheta(w_1)  > }.
\end{equation}
   \begin{thm}
  We have:
\item[(1)] $\widetilde V(-4, D_6)\cong \mathcal {V}_{-4}(D_{6}) .$
\item[(2)] The set $\{  L_{D_6} (-(s+4)\Lambda_0+ s \Lambda _6) \ \vert \ \ \ s \in {\Z}_{\ge 0} \}$ provides  a complete list of irreducible $\mathcal {V}_{-4}(D_{6}) $--modules.
\end{thm}
 \begin{proof}
 We first notice that   $ \widetilde V(-4, D_6) $ is a certain quotient of $\frac{V^{-4}(D_6)} { < v_1 , \vartheta (w_1) > }$, and that  $$H_{\theta}  (\frac{V^{-4}(D_6)} { < v_1 , \vartheta (w_1) > })  = \mathcal V_{-2} (D_4).$$
Since $ \mathcal V_{-2} (D_4)$  contains a unique non-trivial ideal which is maximal and simple, we conclude that $\frac{V^{-4}(D_6)} { < v_1 , \vartheta (w_1) > }$ also contains a unique ideal, and it must be  the ideal generated by $w_1$. Since in $\widetilde V(-4, D_6)$ we have that $w_1 \ne 0$, we conclude that
$$\widetilde V(-4, D_6) \cong  \frac{V^{-4}(D_6)} { < v_1, \vartheta(w_1)  > }.$$
The proof of assertion (2) follows from  (1), the classification result of $\overline  V_{-4} (D_6)$--modules from Proposition  \ref{prop-rep-th-D} and Lemma \ref{pomoc-11}.
 \end{proof}

\section{Funding}

This work was supported by the Croatian Science Foundation
[grant number 2634 to D.A. and O.P.];
and the QuantiXLie Centre of Excellence, a project cofinanced
by the Croatian Government and European Union
through the European Regional Development Fund - the
Competitiveness and Cohesion Operational Programme
[KK.01.1.1.01 to D.A. and O.P.].

\end{document}